\documentclass[10pt,a4paper]{amsart}

\usepackage[utf8]{inputenc}
\usepackage[T1]{fontenc}
\usepackage{amsmath}
\usepackage{amsfonts}
\usepackage{ae}
\usepackage{units}
\usepackage{icomma}
\usepackage{color}
\usepackage{graphicx}
\usepackage{bbm}
\usepackage{amsthm}
\usepackage{amssymb}
\usepackage{cancel}
\usepackage{type1cm}
\usepackage{eso-pic}
\usepackage[breaklinks,pdfpagelabels=false]{hyperref}
\usepackage{xspace}
\usepackage{enumitem}
\usepackage{bbm}
\usepackage{verbatim}
\usepackage{subfig}
\usepackage{tikz}
\usetikzlibrary{arrows}
\usepackage{amsrefs}

\newcommand{\abs}[1]{\ensuremath{\left| #1 \right|}}

\newcommand{\bm}[1]{\mbox{\boldmath $ {#1} $}}
\newcommand{\vz}{\bm{\hat{0}}}
\newcommand{\vo}{\bm{\hat{1}}}
\newcommand{\vv}{\bm{\hat{v}}}

\newcommand{\vvs}{{\hat{v}}}

\newcommand{\sech}{\operatorname{sech}}
\newcommand{\csch}{\operatorname{csch}}
\newcommand{\defeq}{:=}
\newcommand{\hc}[1]{\ensuremath{\mathbb{Q}_{ #1 }}}
\newcommand{\dhc}[1]{\ensuremath{\overrightarrow{\mathbb{Q}_{ #1 }}}}

\renewcommand{\epsilon}{\varepsilon}

\newtheorem{theorem}{Theorem}[section]
\newtheorem{lemma}[theorem]{Lemma}
\newtheorem{proposition}[theorem]{Proposition}

\theoremstyle{definition}
\newtheorem*{definition}{Definition}
\newtheorem{remark}[theorem]{Remark}

\numberwithin{equation}{section}
\numberwithin{theorem}{section}

\begin{document}
\title[Acc. percolation and f.p. site percolation on $\hc{n}$]{Accessibility percolation and first-passage site percolation on the unoriented binary hypercube}
\author{Anders Martinsson}
\address{Department of Mathematical Sciences, Chalmers University Of Technology and University of Gothenburg, 41296 Gothenburg, Sweden}
\email{andemar@chalmers.se}

\begin{abstract}
Inspired by biological evolution, we consider the following so-called accessibility percolation problem: The vertices of the unoriented $n$-dimensional binary hypercube are assigned independent $U(0, 1)$ weights, referred to as fitnesses. A path is considered accessible if fitnesses are strictly increasing along it. We prove that the probability that the global fitness maximum is accessible from the all zeroes vertex converges to $1-\frac{1}{2}\ln\left(2+\sqrt{5}\right)$ as $n\rightarrow\infty$. Moreover, we prove that if one conditions on the location of the fitness maximum being $\vv$, then provided $\vv$ is not too close to the all zeroes vertex in Hamming distance, the probability that $\vv$ is accessible converges to a function of this distance divided by $n$ as $n\rightarrow\infty$. This resolves a conjecture by Berestycki, Brunet and Shi in almost full generality.

As a second result we show that, for any graph, accessibility percolation can equivalently be formulated in terms of first-passage site percolation. This connection is of particular importance for the study of accessibility percolation on trees.
\end{abstract}

\maketitle

\section{Introduction}

A number of recent papers \cites{fkdk11,hm13,bbz1,bbz2,nk13,rz13,c14} have studied a percolation problem known as accessibility percolation, based on ideas of Kauffman and Levin for modeling biological evolution \cite{kl87}. In its simplest form, accessibility percolation consists of a graph $G=(V, E)$, or more generally a digraph, together with a \emph{fitness function} $\omega:V\rightarrow\mathbb{R}$ generated according to some random distribution. This is thought of as representing the landscape of possible evolutionary trajectories of a species. The vertices in $G$ represent the possible genotypes for an organism whose fitness is a measure of how successful an individual of that genotype is, and the edges the possible ways the genome can change subject to a single mutation. Here it makes sense both to consider directed and undirected edges depending on whether or not a certain mutation is reversible. Of primary concern is the existence or distribution of so-called \emph{accessible} paths.
\begin{definition}
Let $G=(V, E)$ and $\omega:V\rightarrow\mathbb{R}$ be a fitness landscape. We say that a path $v_0\rightarrow v_1 \rightarrow\dots\,\rightarrow v_l$ in $G$ is \emph{accessible} if
\begin{equation}
\omega(v_0) < \omega(v_1) <\dots\,< \omega(v_l).
\end{equation}
For $v, w\in V$ we say that \emph{$w$ is accessible from $v$} if there exists an accessible path from $v$ to $w$.
\end{definition}

For the distribution of $\omega$ we will in this paper consider two variations of Kingman's House-of-Cards model \cite{kingman}. Both of which have previously been considered in accessibility percolation.  In fact, all results in \cites{bbz1,bbz2,nk13,rz13,c14} consider some variation of the House-of-Cards model, whereas \cite{fkdk11} and \cite{hm13} also consider the so-called Rough Mount Fuji model. The first model we will consider here is the original formulation of the House-of-Cards model, in which the $\omega(v)$:s are independent and $U(0, 1)$-distributed for all $v\in V$. Kauffman and Levin refers to this as an uncorrelated landscape. For the second distribution we modify the House-of-Cards model by introducing an a priori global fitness maximum $\vv\in V$ by changing $\omega(\vv)$ to one. As accessibility percolation only depends on the relative order of fitnesses, this can be seen as equivalent to conditioning the House-of-Cards model on $\vv$ being the global fitness maximum. In particular, if $\vv$ is chosen uniformly at random among $V$, then this is equivalent to the House-of-Cards model with $\vv$ denoting the global fitness maximum.

Our first main result considers accessibility percolation on the unoriented $n$-dimensional binary hypercube. The question of primary concern is whether or not there exists an accessible path from the all zeroes vertex, $\vz$, to the fitness maximum $\vv$. We prove that, provided $\vv$ is not too close to $\vz$ in Hamming distance, the probability that such path exists converges to a non-trivial function of the Hamming distance between $\vv$ and $\vz$ divided by $n$, confirming a conjecture by Berestycki, Brunet and Shi \cite{bbz2} in almost full generality.

As a second result, we show that accessibility percolation for a general graph can be equivalently formulated in terms of first-passage site percolation. This lets us reformulate previous results in the literature in terms of first-passage site percolation. In particular, this relation has important implications for accessibility percolation on trees, as studied in \cites{bbz1,nk13,rz13,c14}.

\subsection{Notation}
\begin{itemize}
\item Whenever talking about a general graph $G=(V, E)$, we allow both undirected and directed edges. For vertices $u, v\in V$, we write $u\sim v$ if there is either an undirected edge between $u$ and $v$ or a directed edge going from $u$ to $v$.
\item The unoriented $n$-dimensional binary hypercube, denoted by $\hc{n}$, is the graph whose vertices are the binary $n$-tuplets $\{0, 1\}^n$ and where two vertices share an edge if their Hamming distance is one. The oriented $n$-dimensional binary hypercube, $\dhc{n}$, is the directed graph obtained by directing each edge in $\hc{n}$ towards the vertex with more ones.
\item For a vertex $v$ in the hypercube we let $\abs{v}$ denote the number of coordinates of $v$ that are one. Addition and subtraction of vertices in $\hc{n}$ denotes coordinate-wise addition/subtraction modulo two. We let $\vz$ and $\vo$ denote the all zeroes and all ones vertices respectively, and let $e_1, \dots, e_n$ denote the standard basis.
\item Often when considering the House-of-Cards model, it is useful to condition on the fitness of $\vz$. Following the convention in \cites{bbz1,bbz2}, for any $\alpha\in[0, 1]$ we let $\mathbb{P}^\alpha(\cdot)$ and $\mathbb{E}^\alpha\left[\cdot\right]$ denote conditional probability and expectation respectively, given $\omega(\vz)=\alpha$.
\end{itemize}

\subsection{Recent work}

Let us take a moment to summarize the results for accessibility percolation on the binary hypercube with House-of-Cards fitnesses in \cites{hm13,bbz1,bbz2}. We start by consider the simplified version of the problem where we replace $\hc{n}$ by $\dhc{n}$. This is equivalent to only considering paths without backwards mutations. As any coordinate where $\vv$ is zero will be constantly zero along any such path, it suffices to consider the case where $\vv=\vo$.

Let $X$ denote the number of oriented paths from $\vz$ to $\vo$ which are accessible. As there are $n!$ oriented paths from $\vz$ to $\vo$, and each path is accessible if and only if the $n$ random fitnesses along the path are in ascending order, we see that $\mathbb{E}X = 1$. At first glance, this may seem to imply a positive probability of accessible paths existing. However, a much clearer picture of what occurs is obtained by conditioning on the fitness of the starting vertex. Indeed, conditioned on the fitness of $\vz$ being $\alpha\in[0, 1]$, we have
\begin{equation}\label{eq:orientedEX}
\mathbb{E}^{\alpha} X = n (1-\alpha)^{n-1}.
\end{equation}
We see that, for large $n$, this expression is $1$ approximately at $\alpha = \frac{\ln n}{n}$, and rapidly decreasing as $\alpha$ increases. Informally, this means that unless the fitness of the starting vertex is below $\frac{\ln n}{n}$, accessible paths are highly unlikely. In fact, by considering \eqref{eq:orientedEX} a bit more closely it follows that $\mathbb{P}\left(X\geq 1 \wedge \omega(\vz)>\frac{\ln n}{n}\right) \leq \frac{1}{n}$. On the other hand, the regime where $\alpha$ is smaller than $\frac{\ln n}{n}$ turns out to be more difficult to treat. In \cite{hm13} it was shown by Hegarty and the author that the probability of accessible paths in this case tends to $1$ as $n\rightarrow\infty$.
\begin{theorem}\label{thm:hegartymartinsson}(Hegarty, Martinsson)
For any sequence $\{\varepsilon_n\}_{n=1}^\infty$ such that $n\varepsilon_n \rightarrow \infty$, as $n\rightarrow \infty$ we have
\begin{align}
\mathbb{P}^{\frac{\ln n}{n}+\varepsilon_n} (X\geq 1)&\rightarrow 0\\
\mathbb{P}^{\frac{\ln n}{n}-\varepsilon_n} (X\geq 1)&\rightarrow 1.
\end{align}
Furthermore,
\begin{equation}
\mathbb{P}(X\geq 1) \sim \frac{\ln n}{n}.
\end{equation}
\end{theorem}
This theorem was later strengthened by Berestycki, Brunet and Shi in \cite{bbz1} who proved that, in the special case where $\omega(\vz)=O\left(\frac{1}{n}\right)$, $X$ has a non-trivial limit distribution when scaled appropriately.

Let us now switch back to the unoriented hypercube and see how this analysis changes. Again, let $X$ denote the number of accessible paths from $\vz$ to $\vv$. Here, paths are not as combinatorially well-behaved as for the oriented cube, and first moment estimates are not as easy to come by. Nevertheless, in a recent paper by Berestycki, Brunet and Shi \cite{bbz2} it was shown that $\mathbb{E}^\alpha X$ has the following asymptotic behavior:
\begin{theorem}\label{thm:berestyckiEX}(Berestycki, Brunet, Shi)
Let $\alpha \in [0, 1]$ be fixed, and let $\vv=\vv_n\in\hc{n}$ be such that $x \defeq \lim_{n\rightarrow\infty}\abs{\vv_n}/n$ exists. We have that as $n\rightarrow\infty$
\begin{equation}\label{eq:berestyckiEX}
\left(\mathbb{E}^\alpha X\right)^{1/n} \rightarrow \sinh(1-\alpha)^{x} \cosh(1-\alpha)^{1-x}.
\end{equation}
As a consequence, for each $x$ there is a critical value $\alpha^*(x)=1-\vartheta(x)$ for the fitness of $\vz$, given by the unique non-negative solution to
\begin{equation}\label{eq:defofvartheta}
\left(\sinh\vartheta\right)^x \left(\cosh\vartheta\right)^{1-x} = 1,
\end{equation}
such that
\begin{itemize}
\item For $\alpha > 1-\vartheta(x)$, $\mathbb{P}^\alpha \left(X\geq 1\right)$ goes to $0$ exponentially fast as $n\rightarrow\infty$.
\item For $\alpha < 1-\vartheta(x)$, $\mathbb{E}^\alpha X$ diverges exponentially fast as $n\rightarrow\infty$.
\end{itemize}
Hence, the unconditioned probability that $X\geq 1$ is at most $1-\vartheta(x)$.
\end{theorem}
We see a similar behavior of $\mathbb{E}^\alpha X$ as for the oriented cube. One important difference though is that unlike the oriented cube the critical value has a nontrivial limit as $n\rightarrow\infty$. The function $\vartheta(x)$ is plotted in Figure \ref{fig:varthetax}. This function is continuous and increasing where $\vartheta(0)=0$ and $\vartheta(1)=\ln\left(1+\sqrt{2}\right)\approx 0.88$. In particular, it follows that if the a priori global fitness maximum is $\vo$, then the critical fitness is $1-\ln\left(1+\sqrt{2}\right)\approx 0.12$, and if chosen uniformly at random then $\abs{\vv}/n$ will be tightly concentrated around $\frac{1}{2}$ and hence the critical fitness is $1-\frac{1}{2}\ln\left(2+\sqrt{5}\right)\approx 0.28$.

\begin{figure}[!ht]
\centering
\includegraphics[width=\textwidth]{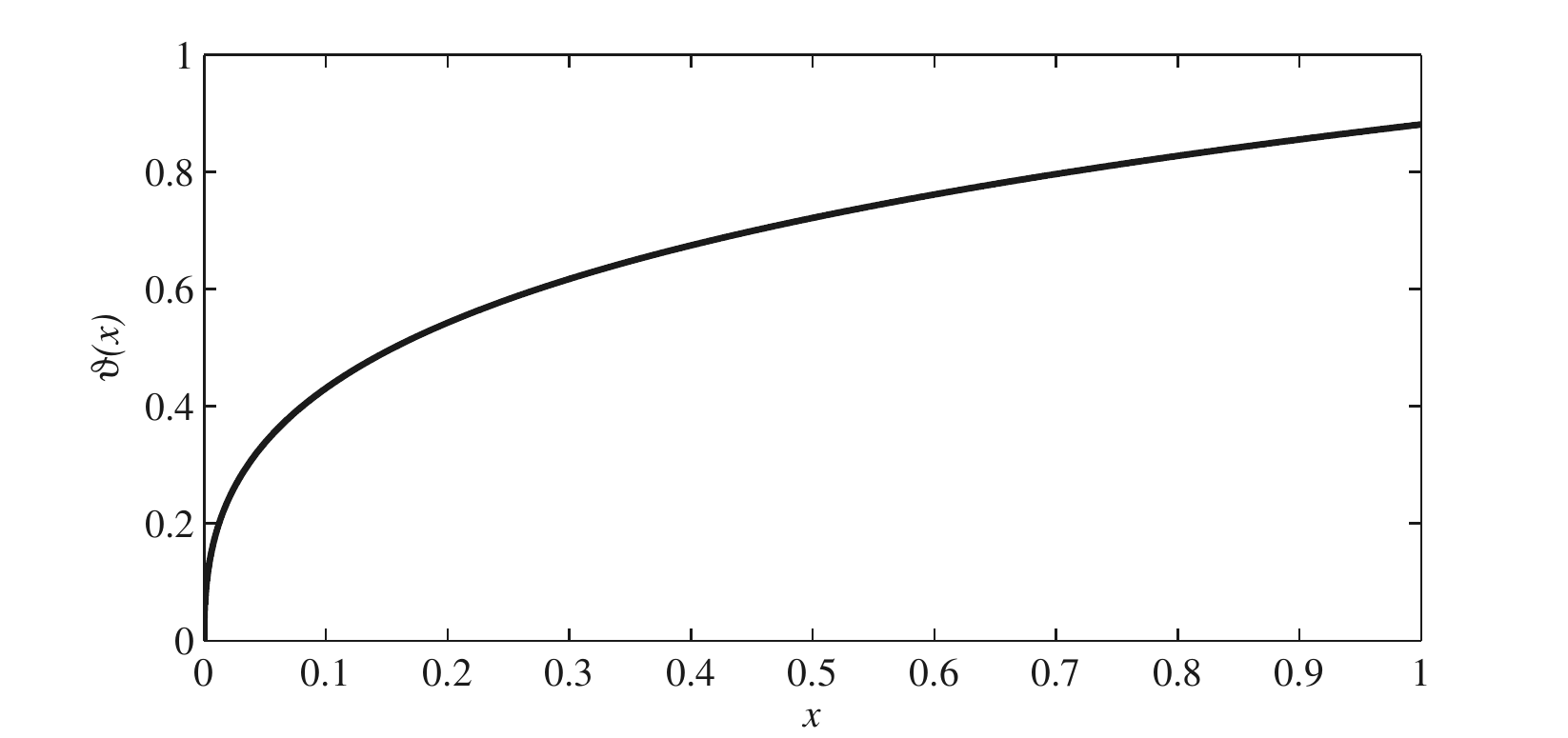}
\caption{The function $\vartheta(x)$ as defined in \eqref{eq:defofvartheta}.}\label{fig:varthetax}
\end{figure}

Berestycki et al. further gave two conjectures that \eqref{eq:berestyckiEX} ``tells the truth'' in the sense that $\mathbb{P}^\alpha\left(X \geq 1\right)$ tends to $1$ as $n\rightarrow\infty$ for $\alpha<1-\vartheta(x)$. Conjecture 1 of their paper proposes this in the special case where $\vv=\vo$, and Conjecture 2 in the more general setting of $\vv=\vv_n$ satisfying $\abs{\vv_n}/n \rightarrow x \in [0, 1]$.

\subsection{Results}
The first result of this paper fully resolves Conjecture 1 by Berestycki, Brunet and Shi \cite{bbz2}, and Conjecture 2 under the additional condition that $x$ is not too small.
\begin{theorem}\label{thm:mainresult1}
Let $\vv=\vv_n\in\hc{n}$ be a sequence of vertices such that $x\defeq\lim_{n\rightarrow\infty}\abs{\vv}/n$ exists. Let $X$ denote the number of accessible paths from $\vz$ to $\vv$. Let $\vartheta(x)$ be as defined in Theorem \ref{thm:berestyckiEX}. Assuming $x\geq 0.002$, we have
\begin{equation}
\lim_{n\rightarrow\infty} \mathbb{P}^\alpha \left(X\geq 1\right) = \begin{cases} 0 &\text{ if } \alpha > 1-\vartheta(x)\\
1 &\text{ if } \alpha < 1-\vartheta(x).
\end{cases}
\end{equation}
In particular, if $\vv=\vo$, then
\begin{equation}
\mathbb{P}\left(X\geq 1\right) \rightarrow 1-\ln\left(1+\sqrt{2}\right) \text{ as }n\rightarrow\infty
\end{equation}
and if $\vv$ is chosen uniformly at random, then
\begin{equation}
\mathbb{P}\left(X\geq 1\right) \rightarrow 1-\frac{1}{2}\ln\left(2+\sqrt{5}\right) \text{ as }n\rightarrow\infty.
\end{equation}
\end{theorem}
The value $0.002$ deserves some explanation. In the proof of Theorem \ref{thm:mainresult1}, or more accurately the proof of Theorem \ref{thm:mainresult1equiv} below which is shown to be equivalent to the former, we see that there is a value $x^* \approx 0.00167$ such that the proof goes through whenever $x>x^*$ and breaks down when $x<x^*$, see Remark \ref{rem:breakdown}. It seems likely however that this is simply an artifact of the technique used in the proof, and that the statement should hold true even for smaller $x$. Regardless of whether or not this is true, we can note that the two cases of most concern, $x=1$ and $x=0.5$, are far above $x^*$.

We now turn to the relation between accessibility percolation and first-passage site percolation for a general graph. Let $G=(V, E)$ be a graph with a distinguished vertex $\vz$. Note that each edge of $G$ may either be directed or undirected. For each vertex $v\in G$ randomly assign a cost, denoted by $c(v)$, according to independent $U(0, 1)$ random variables. For a path $u_0, u_1, \dots, u_l$ in $G$ we define the site passage time of the path by
\begin{equation}
\sum_{1\leq i \leq l} c(u_i),
\end{equation}
and similarly define its reduced site passage time by
\begin{equation}
\sum_{1\leq i < l} c(u_i).
\end{equation}
Note that neither the passage time nor the reduced passage time of a path include the cost of the first vertex. For each $u, v\in G$ we define the site first-passage time from $u$ to $v$, denoted by $\mathcal{T}_V(u, v)$, and the reduced first passage time from $u$ to $v$, denoted by $\mathcal{T}_V'(u, v)$, as the minimum of the respective quantity over all paths from $u$ to $v$.

\begin{theorem}\label{thm:mainresult2}
Let $G$ be a graph with two distinct vertices $\vz$ and $\vv$, and let $\alpha \in [0, 1]$. Consider accessibility percolation on $G$. If fitnesses are assigned according to the House-of-Cards model with $\vv$ as the a priori global fitness maximum, then
\begin{equation}
\mathbb{P}^{\alpha}\left( \vv \text{ accessible from }\vz \right) = \mathbb{P}\left( \mathcal{T}'_V(\vz, \vv) \leq 1-\alpha\right).
\end{equation}
If fitnesses are assigned according to the House-of-Cards model without an a priori global maximum, then for any vertex $v \in G$
\begin{equation}
\mathbb{P}^{\alpha}\left( v \text{ accessible from }\vz \right) = \mathbb{P}\left( \mathcal{T}_V(\vz, v) \leq 1-\alpha\right).
\end{equation}
Moreover, in the latter case this claim can be significantly strengthened. Conditioned on the fitness of $\vz$ being $\alpha$, the set of vertices accessible from $\vz$ has the same distribution as the set of vertices $v$ such that $\mathcal{T}_V(\vz, v)\leq 1-\alpha$.
\end{theorem}
Informally we can think of this theorem as saying that accessibility percolation is equivalent to first-passage site percolation with independent $U(0, 1)$ vertex passage times. We need to be a bit careful there though; the theorem only deals with the question of whether or not a certain vertex is accessible from $\vz$ along any path, and it does not for instance say anything about the number of accessible paths. Indeed, it is not true in general that the number of accessible paths from $\vz$ to $v$ is distributed as the number of paths from $\vz$ to $v$ with reduced passage time at most $1-\alpha$. For graphs containing non-simple paths this is clear as non-simple paths can have arbitrarily small passage time but cannot be accessible, but it can even be false for directed acyclic graphs, see for instance Figure \ref{fig:smownam}. On the other hand, the connection is more general than just treating which vertices are accessible. For instance, using the proof ideas in Section \ref{sec:proofbondsitemainresult2} one can show that the minimal number of times you need to move to a less fit vertex to get from $\vz$ to $v$ is distributed as the integer part of $\mathcal{T}_V(\vz, v)+\alpha$.

\begin{figure}
\captionsetup[subfigure]{labelformat=empty}

\begin{center}
\begin{tikzpicture}[->,>=stealth',shorten >=0pt,auto,node distance=1.5cm,thick,main node/.style={circle,fill=black,draw,minimum size=4pt,inner sep=0pt}]

\node[main node] (a) [label=below:$\vz$] {};
\node[main node] (b1) [above right of=a] {};
\node[main node] (b2) [below right of=a] {};
\node[main node] (c) [below right of=b1] {};
\node[main node] (d1) [above right of=c] {};
\node[main node] (d2) [below right of=c] {};
\node[main node] (e) [below right of=d1] [label=below:$v$] {};

\path
(a) edge (b1)
(a) edge (b2)
(b1) edge (c)
(b2) edge (c)
(c) edge (d1)
(c) edge (d2)
(d1) edge (e)
(d2) edge (e);
\end{tikzpicture}

\caption{Example of a graph where accessible paths have a different distribution than paths with small passage time. We can for instance note that there can never be exactly three accessible paths from $\vz$ to $v$, whereas there can certainly be exactly three paths with reduced passage time at most $1-\alpha$.
}\label{fig:smownam}
\end{center}
\end{figure}
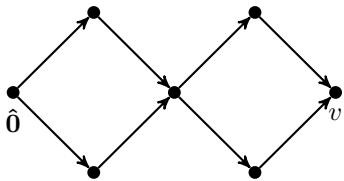

A problem with using Theorem \ref{thm:mainresult2} to relate known results from first-passage percolation to accessibility percolation is that the vast majority of the first-passage percolation literature assigns passage times to edges rather than vertices. However, a common property for percolation problems is that it is harder to percolate on vertices than edges \cite{grimmett}. The following proposition shows that something similar holds for first-passage percolation.
\begin{proposition}\label{prop:bondsite}
Suppose the edges of $G$ are assigned independent $U(0, 1)$ weights. Let $\mathcal{T}_E(u, v)$ denote the minimum total weight of any path from $u$ to $v$ in $G$. Then, it is possible to couple $\mathcal{T}_E(\vz, \cdot)$ to $\mathcal{T}_V(\vz, \cdot)$ such that $\mathcal{T}_E(\vz, v) \leq \mathcal{T}_V(\vz, v)$ for all $v\in G$.
\end{proposition}

For the special case when $G$ is a rooted tree one can see that this coupling is exact; to go from site to bond percolation we can simply consider the passage time of each vertex to instead be assigned to the edge leading to it. Accessibility percolation on trees has been considered in \cites{bbz1, nk13,rz13,c14}. With the exception of \cite{bbz1}, these articles have considered regular rooted trees with degree $n$ and height $h$, and where fitnesses are assigned according to the House-of-Cards model conditioned on the fitness of the root being zero. Of principal concern is how the number of vertices in generation $h$ that are accessible from the root varies as a function of $n$, and in particular whether this number is non-zero. Using Theorem \ref{thm:mainresult2} we can see that this is equivalent to assigning independent $U(0, 1)$ passage times to the edges of the tree and considering the number of vertices $v$ in generation $h$ such that $\mathcal{T}_E(\vz, v) \leq 1$. In particular, the question of whether generation $h$ is accessible from $\vz$ is equivalent to to asking if the first-passage time from the root to generation $h$ is at most $1$. It should be mentioned however that the usual setting in first-passage percolation on regular rooted trees keeps $n$ fixed and considers the first-passage time from the root to generation $h$ as $h\rightarrow\infty$. While the author is not aware of any results from this field that have appropriate error bounds to be directly applicable to accessibility percolation, there seems to be a significant overlap of ideas between \cites{rz13, c14} and the literature on first-passage percolation on trees. See for instance \cite{abr09}. 

Let us now consider the implications of Theorem \ref{thm:mainresult2} for the hypercube. Using this result, we can immediately translate the result from Theorem \ref{thm:hegartymartinsson} to that, for the oriented hypercube, $\mathcal{T}'_V(\vz, \vo)$ is concentrated around $1-\frac{\ln n}{n}$ with fluctuations of order $\frac{1}{n}$. More importantly, we have that the following is equivalent to Theorem \ref{thm:mainresult1}:
\begin{theorem}\label{thm:mainresult1equiv}
Let $G=\hc{n}$ and let $\vv=\vv_n\in\hc{n}$ be a sequence of vertices such that $x\defeq\lim_{n\rightarrow\infty}\abs{\vv}/n$ exists. Assuming $x \geq 0.002$, as $n\rightarrow\infty$ we have
\begin{equation}
\mathcal{T}_V'(\vz, \vv) \rightarrow \vartheta(x)
\end{equation}
in probability.
\end{theorem}
Note here that the fact that $\vartheta(x)$ is an asymptotic lower bound on the reduced passage time is already implied by Theorem \ref{thm:berestyckiEX}.

It should be mentioned that basically the same results holds true for bond percolation. In \cite{fp93} it was shown that for the oriented hypercube, we have $\mathcal{T}_E(\vz, \vo)\rightarrow 1$ in probability as $n\rightarrow\infty$. In a more recent result by the author \cite{m14}, it was shown that for the unoriented hypercube $\mathcal{T}_E(\vz, \vo) \rightarrow \ln\left(1+\sqrt{2}\right)$ as $n\rightarrow\infty$. Strictly speaking these results assume standard exponential edge weights, but it is not too hard to show that the limiting distribution of $\mathcal{T}_E(\vz, \vo)$ only depends on the weight distribution as the righthand limit of its probability distribution function at $0$, hence it will be the same for $U(0, 1)$ weights.

The remainder of the paper will be structured as follows: In Section \ref{sec:proofbondsitemainresult2} we prove Proposition \ref{prop:bondsite} and Theorem \ref{thm:mainresult2}. The remaining sections, Sections \ref{sec:uncontestedCTP}, \ref{sec:calculus} and \ref{sec:bootstrap}, are dedicated to the proof of Theorem \ref{thm:mainresult1equiv}.

\section{Proof of Proposition \ref{prop:bondsite} and Theorem \ref{thm:mainresult2} }\label{sec:proofbondsitemainresult2}

We may, without loss of generality, assume that for any vertex $v$ there exists a path from $\vz$ to $v$.

A key idea of the proofs of Proposition \ref{prop:bondsite} and Theorem \ref{thm:mainresult2} is the following procedure for computing $\mathcal{T}_V(\vz, v)$. We initially consider $\mathcal{T}_V(\vz, v)$ to be unassigned for each $v$, except $\vz$ for which it is set to $0$, and we let $U=\{\vz\}$ denote the set of vertices with assigned first-passage times. Until $\mathcal{T}_V(\vz, v)$ is assigned for all $v$, we do the following operation:
\begin{enumerate}
\item Find a pair of vertices $u, v$ that minimizes $\mathcal{T}_V(\vz, u)$ subject to $(u, v)\in E$, $u\in U$ and $v\not\in U$.
\item Let $\mathcal{T}_V(\vz, v) \defeq \mathcal{T}_V(\vz, u) + c(v)$
\item Add $v$ to $U$.
\end{enumerate}
To see that this assigns first-passage times correctly, suppose that we are in the step where $\mathcal{T}_V(\vz, v)$ is assigned. As $v$ is not in $U$, the passage time of any path from $\vz$ to $v$ must include the passage time from $\vz$ to some vertex $u'$ in $U$ adjacent to some vertex outside $U$, as well as the cost $v$. Hence $\mathcal{T}_V(\vz, v)\geq \mathcal{T}_V(\vz, u') + c(v) \geq \mathcal{T}_V(\vz, u) + c(v)$. As there is a path from $\vz$ to $v$ with passage time $\mathcal{T}_V(\vz, u) + c(v)$, this must be optimal. Hence, if all previous assignments are correct, $\mathcal{T}_V(\vz, v)$ will be assigned correctly as well.

\begin{proof}[Proof of Proposition \ref{prop:bondsite}.]
We can modify this algorithm to run on first-passage bond percolation by replacing $c(v)$ by the weight of the edge from $u$ to $v$. In either case, as no vertex cost or edge weight respectively is accessed more than once, the accessed values form a sequence of independent and $U(0, 1)$ random variables. Hence the distribution of $\mathcal{T}_V(\vz, v)$ is unaffected. On the other hand, for bond percolation we get that $\mathcal{T}_V(\vz, v)$ is the edge passage time of some path from $\vz$ to $v$ (but not necessarily the shortest).
\end{proof}

We now turn to the proof of Theorem \ref{thm:mainresult2}. The coupling between first-passage site percolation and accessibility percolation we will consider is essentially to let $f(v)=\left\{\alpha+\mathcal{T}_V(\vz, v)\right\}$ be the fitness function, where $\{x\}=x-\lfloor x \rfloor$ denotes the fractional part of $x$. We will however modify this slightly by putting $f(v)=1$ whenever $\alpha+\mathcal{T}_V(\vz, v)=1$. It is clear that the probability of such $v$ other than $\vz$ existing is $0$, so the only way this will change the distribution of $f$ is that $f(\vz)=1$ if $\alpha=1$.

It is not too hard to see that, for any vertex $v$ except $\vz$, $f(v)$ is $U(0, 1)$-distributed. The following lemma shows that the $f(v)$:s are also independent, hence showing that $f$ is distributed according to the House-of-Cards model without an a priori global fitness maximum, conditioned on $f(\vz)=\alpha$.
\begin{lemma}\label{lemma:alphaequiv3}
$f(v)$ are independent $U(0, 1)$ random variables for $v \in V \setminus \{\vz\}$.
\end{lemma}
\begin{proof}
Suppose that we generate vertex costs in the following way: Run the procedure above, but with the modification that whenever the algorithm tries to access $c(v)$, first generate a $U(0, 1)$ random variable $\tilde{f}(v)$ and assign $c(v)$ the value $\left\{ \tilde{f}(v) - \alpha - \mathcal{T}_V(\vz, u) \right\}$.

It is clear that the $c(v)$:s are independent and $U(0, 1)$-distributed. The lemma follows by noting that, in the latter case, we have $f(v) = \tilde{f}(v)$ almost surely for all $v\in V\setminus\{\vz\}$.
\end{proof}

\begin{proof}[Proof of Theorem \ref{thm:mainresult2}.]
We begin by considering the case with no a priori global fitness maximum. In this case, we can consider $f:V\rightarrow\mathbb{R}$ to be the fitness function. For simplicity let us assume that no vertex cost is exactly $0$.

Assume $\mathcal{T}_V(\vz, v) \leq 1-\alpha$, and let $\vz=v_0, v_1, \dots, v_l=v$ be the path with shortest passage time. Then, as $0 < \alpha+\mathcal{T}_V(\vz, v_i) \leq 1$ for $1\leq i \leq l$ it follows that
\begin{equation}
f(v_i) = \alpha + \sum_{j=1}^{i} c(v_j)
\end{equation}
for $0 \leq i \leq l$. Hence $v_0, v_1, \dots, v_l$ is accessible. Conversely, suppose $\mathcal{T}_V(\vz, v) > 1-\alpha$ and let $\vz=v_0, v_1, \dots, v_l=v$ be any path between $\vz$ and $v$. Let $i$ be the lowest index such that $\mathcal{T}_V(\vz, v_i) > 1-\alpha$. Then $f(v_i) \leq \alpha+\mathcal{T}_V(\vz, v_i)-1 \leq \alpha+\mathcal{T}_V(\vz, v_{i-1})+c(v_i)-1 < \alpha+\mathcal{T}_V(\vz, v_{i-1}) = f(v_{i-1})$. Hence the path is not accessible.

Now for the case where $\vv\in V\setminus\{\vz\}$ is the a priori global fitness maximum. We here keep the same coupling as before between $f(v)$ and $c(v)$ for $v\in V$, except that we fix $f(\vv)=1$. Let $U$ be the set of vertices $v\in V$ such that $(v, \vv) \in E$. Then $\vv$ being accessible from $\vz$ is almost surely equivalent to some vertex in $U$ being accessible from $\vz$. Note that this last statement does not depend on the value of $f(\vv)$. It follows that $\vv$ is accessible from $\vz$ is almost surely equivalent to that $\min_{v\in U} \mathcal{T}_V(\vz, v) \leq 1-\alpha$. The theorem follows by noting that $\min_{v\in U} \mathcal{T}_V(\vz, v) = \mathcal{T}_V'(\vz, \vv)$.
\end{proof}

\section{The Clustering Translation Process}\label{sec:uncontestedCTP}
Before proceeding, we will slightly modify $\mathcal{T}'_V(\vz, \vv)$ by replacing the $U(0, 1)$ vertex costs by independent standard exponential such. Note that the standard exponential distribution stochastically dominates $U(0, 1)$, and hence this modification will only increase $\mathcal{T}'_V(\vz, \vv)$. As the lower bound in Theorem \ref{thm:mainresult1equiv} follows from Theorem \ref{thm:berestyckiEX}, it suffices to show that, with this modification, asymptotically almost surely $\mathcal{T}'_V(\vz, \vv)\leq \vartheta(x)+o(1)$. To do this, we will mimic the argument in \cite{m14} for first-passage bond percolation on $\hc{n}$.

Let us take a moment to describe some of the underlying machinery for first-passage bond percolation on $\hc{n}$. We assume independent standard exponential edge weights. In \cite{fp93}, Durrett introduced the following process, which he called the the \emph{branching translation process}, BTP: At time $0$ we place one particle at $\vz$ in $\hc{n}$. The system then evolves by each existing particle independently generating offspring at each vertex adjacent to its position at rate $1$. One can show that for each vertex $v\in\hc{n}$, the time at which the first particle at $v$ is born is stochastically dominated by $\mathcal{T}_E(\vz, v)$. This follows from the fact that the BTP dominates the so-called Richardson's model. The strategy in \cite{m14} is basically to show that, with a certain coupling, there is a probability bounded away from zero of these quantities being equal.

In order to translate this approach to first-passage site percolation, we need to find a corresponding process to the BTP for this case. We claim that the following is such a process: We initially have a finite number of particles, each located at a vertex in $\hc{n}$. For each particle, we assign an independent Poisson clock with unit rate. When a particle's clock goes off, it simultaneously generates one new offspring at each vertex adjacent to its position. The new particles are then assigned new Poisson clocks and the process continues. We will refer to this process as the clustering translation process, CTP.

We see that in both the BTP and CTP each particle generates offspring at each neighboring vertex at rate $1$. A big difference however is that in the BTP this is done independently for each neighboring vertex, whereas in the CTP a particle generates offspring all neighboring vertices simultaneously. Another difference is that the initial state of the CTP is not fixed.

The most important initial state of the CTP will be one particle at each neighbor of $\vz$. We will refer to a CTP initialized in this way as a \emph{standard} CTP. Particles born due to the same Poisson clock tick will be referred to as \emph{identical $n$-tuplets}. To simplify terminology we will also consider the initial $n$ particles in a standard CTP as identical $n$-tuplets. Below we will use the terms \emph{ancestor} and \emph{descendant} of a particle to denote the natural partial order of particles generated by the CTP. For convenience, we say that a particle is both an ancestor and a descendant of itself. The terms \emph{parent} and \emph{child} are defined in the natural way. The $\emph{ancestral line}$ of a particle $x$ is the ordered set of ancestors of $x$, and we say that the ancestral line of $x$ \emph{follows the path $\vz=v_0, v_1, v_2, \dots, v_l$} if the location of the ancestors of $x$ in chronological order is given by $v_1, v_2, \dots v_l$. Note that this path always starts at $\vz$ even though the first ancestor is located at a neighbor of $\vz$. We say that a particle $x$ \emph{originates from a particle $y$ at a time $t$} if $y$ is the last particle in the ancestral line of $x$ that exists at time $t$.

We can immediately note some properties of this process. Firstly, it is Markovian. Secondly, let $A$ be a set of vertices in $\hc{n}$, and let $M_A(v, t)$ denote the expected number of particles at vertex $v$ at time $t\geq 0$ in the CTP initialized by placing one particle at each vertex in $A$. Then it is easy to see that $M_A(v, t)$ must solve the initial value problem
\begin{align}
\frac{d}{dt} M_A(v, t) &= \sum_{w \sim v} M_A(w, t) \text{ for }t>0\label{eq:MAdeq}\\
M_A(v, 0) &= \mathbbm{1}_A(v).
\end{align}
In particular, if $A=\{\vz\}$, then the unique solution to this problem is 
\begin{equation}
m(v, t) \defeq \left(\sinh t\right)^{\abs{v}} \left(\cosh t\right)^{n-\abs{v}},
\end{equation}
and it follows by linearity that for any $A$, we have
\begin{equation}
M_A(v, t) = \sum_{w\in A} m(v-w, t).
\end{equation}
Recall that addition/subtraction of vertices in $\hc{n}$ are interpreted as coordinate-wise addition/subtraction modulo $2$. It should be remarked that the exact same analysis holds for the BTP.

We now show that the standard CTP indeed has the desired relation to first-passage site percolation. To this end, we partition the particles in this process into two sets, the set of \emph{alive} particles and the set of \emph{ghosts}. Each initial particle is alive. Whenever a new particle is born, it is alive if its location does not already contain an alive particle and its parent is alive, and is a ghost otherwise. Note that at most one particle at each vertex can be alive. Furthermore, it is easy to show that each vertex will almost surely eventually contain an alive particle.
\begin{proposition}\label{prop:fppisaliveCTP}
Consider first-passage site percolation on $\hc{n}$ with exponentially distributed costs with unit mean. It is possible to couple this process to the standard CTP such that for each vertex $v$ except $\vz$, $\mathcal{T}_V'(\vz, v)$ denotes the birth time of the alive particle at $v$.
\end{proposition}
\begin{proof}
For each vertex $v$, we let $\tilde{\mathcal{T}}'(v)$ denote the first time $t\geq 0$ when $v$ contains an alive particle, and we let $\tilde{c}(v)$ denote the time from the birth of this particle to the first arrival of its clock. Then $\tilde{c}(v)$ for $v\in \hc{n}$ are independent exponentially distributed random variables with unit mean. 

From the definitions of the CTP and alive particles, it follows that for any vertex $v$ that is not a neighbor of $\vz$, the alive particle at $v$ is born at the first arrival time of an alive particle at an adjacent vertex. Hence, for any $v$ that is not a neighbor of $\vz$, we have
\begin{equation}
\tilde{\mathcal{T}}'(v) = \min_{w\sim v} \left(\tilde{\mathcal{T}}'(w) + \tilde{c}(w)\right),
\end{equation}
and trivially $\tilde{\mathcal{T}}'(v)=0$ when $v$ is a neighbor of $\vz$. It is easy to see that this uniquely defines $\tilde{\mathcal{T}}'(v)$, and that for each vertex $v$ except $\vz$, $\tilde{\mathcal{T}}'(v)$ denotes the reduced first-passage time from $\vz$ to $v$ with respect to the vertex costs given by $\tilde{c}(v)$.
\end{proof}

Given this proposition, we are able to proceed analogously to Sections 2 and 3 in \cite{m14}. In applying this coupling between the CTP and first-passage site percolation we will consider a stronger and more tractable property than aliveness. For any particle $x$ in the CTP, we let $c(x)$ denote the number of pairs of particles $y$ and $z$ such that
\begin{itemize}
\item $y$ and $z$ occupy the same vertex
\item $y$ is an ancestor of $x$
\item $y$ was born after $z$.
\end{itemize}
We furthermore let $a(x)$ denote the number of such pairs where $z$ is either an ancestor of $x$ or an identical $n$-tuple of an ancestor of $x$, and define $b(x)=c(x)-a(x)$. We call a particle $x$ uncontested if $c(x)=0$.

It can be noted that $a(x)$ is defined differently for the BTP. This is because the strategy is loosely speaking to let $a(x)$ denote the number of pairs $(y, z)$ that deterministically must exist given $x$. For the CTP we have additional such pairs, namely those corresponding to identical $n$-tuplets of ancestors of $x$.

\begin{lemma}\label{lemma:uncontestedisalive}
If a particle is uncontested, then it is alive.
\end{lemma}
\begin{proof}
If a particle $x$ is a ghost, then it must have an earliest ancestor (possibly itself) which is a ghost, $y$. As $y$ is a ghost but the parent of $y$ is alive, it follows that the location of $y$ must have already been occupied by some (alive) particle $z$. The pair $(y, z)$ is then counted in $c(x)$.
\end{proof}

It is not hard to see that $a(x)$ only depends on the path followed by the ancestral line of $x$. If we know this path, then we know the locations and order of births of all ancestors and identical $n$-tuplets of ancestors of $x$. Let $\sigma$ be a path represented as a vertex sequence. We say that $\sigma$ is vertex-minimal if there is no proper subsequence which is a path with the same end points.

\begin{lemma}\label{lemma:vertexminimal}
Let $x$ be a particle in the CTP. If the ancestral line of $x$ is vertex-minimal, then $a(x)=0$. The converse is true unless $x$ is located at $\vz$.
\end{lemma}
\begin{proof}
Denote the path followed by the ancestral line of $x$ by $v_0, v_1, \dots, v_l$ and the ancestors of $x$ by $x_1, x_2, \dots, x_l=x$. We have that $a(x)>0$ if and only if there exist $1 \leq i < j \leq l$ such that $x_j$ occupies the same vertex as either $x_i$ or an identical $n$-tuplet of $x_i$, that is, $v_{i-1}$ and $v_j$ are adjacent. Hence, if $a(x)>0$ the path is not vertex-minimal. Conversely, if $a(x)=0$ it follows that the only pairs of adjacent vertices are consecutive in the path. It is straight-forward to show that, unless the path starts and stops at the same vertex, this implies vertex-minimality.
\end{proof}

What follows are two technical lemmas, corresponding to Lemmas 3.2 and 3.3 in \cite{m14}. Before presenting these, we need to specify how to formally describe the CTP. Firstly, by a (potential) particle we mean a word $\{v_1,\,z_1,\,v_2,\,z_2,\dots\,,v_{l-1},\,z_{l-1},\,v_{l}\}$ where $v_1, \dots, v_l$ denote vertices and $z_1,\dots\,z_{l-1}$ positive real numbers. This is interpreted as the particle whose ancestors are located at $v_1, v_2, \dots, v_l$ and born at times $0,\,z_1,\,z_1+z_2$ and so on. The CTP is described by a random set $\mathbf{X}$ of potential particles, denoting the set of particles that will ever be born in the CTP. We will use $\oplus$ to denote concatenation of words. We remark that this representation means that the functions $c(x)$ and $b(x)$ are not functions only of $x$, and should more correctly be denoted by $c(\mathbf{X}, x)$ and $b(\mathbf{X}, x)$. On the other hand, $a(x)$ is really a function of $x$ as it only depends on the location of the ancestors of $x$.

\begin{lemma}\label{lemma:palmCTP}
Let $\sigma = \{\vz=v_0, v_1, \dots, v_{l-1}, v_l \}$ be a path. For $0 \leq i \leq l-1$ let $\mathbf{X}_i$ denote independent CTP:s where $\mathbf{X}_i$ is the CTP obtained by initially placing one particle at each neighbor of $v_i$. Let $f$ be a function that maps pairs $(X, x)$ to the non-negative real numbers where $X$ is a realization of a CTP, and $x$ is a particle in $X$. Similarly, let $V_\sigma(X)$ denote the set of particles in $X$ whose ancestral lines follow $\sigma$. Then for a standard CTP, $\mathbf{X}$, we have
\begin{equation}\label{eq:palm}
\mathbb{E} \sum_{x\in V_\sigma(\mathbf{X}) } f(\mathbf{X}, x) = \int_0^\infty \dots \int_0^\infty \mathbb{E}f\left( \mathbf{X}^{z_1,\dots z_{l-1}}, x^{z_1,\dots,z_{l-1}}\right)\,dz_1\dots\,dz_{l-1},
\end{equation}
where
\begin{equation}
\mathbf{X}^{z_1,\dots z_{l-1}} = \mathbf{X}_0 \cup \left( \{v_1 z_1\}\oplus \mathbf{X}_1  \right) \cup \dots \cup \left( \{v_1 z_1 v_2 z_2 \dots v_{l-1} z_{l-1} \}\oplus \mathbf{X}_{l-1} \right)
\end{equation}
and $x^{z_1,\dots,z_{l-1}} = \{v_1 z_1 v_2 z_2 \dots v_{l} \}$.
\end{lemma}
For compactness, we will only sketch a proof. The reader unconvinced by this is referred to the proof of Lemma 3.2 in \cite{m14}.
\begin{proof}[Proof sketch.]
Let us first consider the case when $f(X, x)$ only depends on $x$. In that case, we have
\begin{equation}
\mathbb{E} \sum_{x\in V_\sigma(\mathbf{X}) } f(x) = \int_0^\infty \dots \int_0^\infty f\left( x^{z_1,\dots,z_{l-1}}\right) \,dz_1\dots\,dz_{l-1}.
\end{equation}
This is because the original particle at $v_1$ gives birth to particles at $v_2$ at rate one whereupon, after its birth, each child at $v_2$ of this original particle gives birth to particles at $v_3$ at rate one, and so on. When $f$ also depends on the realization of the CTP, the idea is that we substitute $f(\mathbf{X}, x)$ in the left-hand side of this sum by $\mathbb{E}\left[ f(\mathbf{X}, x)\middle\vert x\in \mathbf{X}\right]$. Now, formally this conditioning does not really make sense, but its meaning is intuitively clear; it denotes the average value of $f(\mathbf{X}, x)$ where the average is taken over all $\mathbf{X}$ that include $x$. We have
\begin{align*}
&\mathbb{E} \sum_{x\in V_\sigma(\mathbf{X}) } f(\mathbf{X}, x) = \mathbb{E} \sum_{x\in V_\sigma(\mathbf{X}) } \mathbb{E}\left[ f(\mathbf{X}, x)\middle\vert x\in \mathbf{X}\right]\\
&\qquad = \int_0^\infty \dots \int_0^\infty \mathbb{E}\left[ f(\mathbf{X}, x^{z_1,\dots,z_{l-1}})\middle\vert x^{z_1,\dots,z_{l-1}}\in \mathbf{X}\right]\,dz_1\dots\,dz_{l-1}.
\end{align*}
Now, $x^{z_1,\dots,z_{l-1}}$ exists in $\mathbf{X}$ if and only if certain Poisson clocks have arrivals at certain times. By the independent increment property, conditioning on these arrivals does not affect the Poisson clocks at any other times. Hence, the conditional distribution of $\mathbf{X}$ given the existence of $x^{z_1,\dots,z_{l-1}}$ is the same as that of a standard CTP, except with added arrivals, corresponding to the births of the ancestors of $x^{z_1,\dots,z_{l-1}}$. This is precisely the distribution of $\mathbf{X}^{z_1,\dots z_{l-1}}$.
\end{proof}

\begin{lemma}\label{lemma:poissonCTP}
Let $\mathbf{X}$ be a CTP, and let $\phi$ be an indicator function on the set of potential particles in $\mathbf{X}$. If $\phi(x)=0$ for all original particles in the CTP, then
\begin{equation}
\mathbb{P}\left( \sum_{x\in \mathbf{X}} \phi(x)=0\right) \geq \exp\left(-\mathbb{E}\sum_{x\in \mathbf{X}} \phi(x)\right).
\end{equation}
\end{lemma}
\begin{proof}
Let us refer to the set of original particles as generation one, their children as generation two and so on. Let $\mathbf{T}$ denote the set of birth times for particles in generation two in $\mathbf{X}$, and let $\mathbf{T}'\subseteq \mathbf{T}$ be the subset obtained by including $t\in \mathbf{T}$ if there exists a particle $x\in \mathbf{X}$ such that $\phi(x)=1$ and $x$ is an descendant of a particle in generation two born at time $t$. It is clear that $\abs{\mathbf{T}'} \leq \sum_{x\in \mathbf{X}} \phi(x)$ and that $\abs{\mathbf{T}'}=0$ if and only if $\sum_{x\in \mathbf{X}} \phi(x)=0$.

By definition of the CTP, it is clear that $\mathbf{T}$ is a Poisson point process. Furthermore, as the event that $t\in \mathbf{T}$ is included in $\mathbf{T}'$ only depends on descendants of particles in generation two born at time $t$, this occurs independently for each $t\in \mathbf{T}$. Hence, by the random selection property, $\mathbf{T}'$ is also a Poisson point process. This implies that
\begin{equation}
\mathbb{P}\left( \sum_{x\in \mathbf{X}} \phi(x)=0 \right) = \mathbb{P}\left( \mathbf{T}'=\emptyset \right) = \exp\left(-\mathbb{E}\abs{\mathbf{T}'}\right)\geq \exp\left(-\mathbb{E}\sum_{x\in\mathbf{X}} \phi(x)\right),
\end{equation}
as desired.
\end{proof}

\begin{theorem}\label{thm:uncontestedCTP}
Consider a standard CTP. For any vertex $v$ and any $t\geq 0$, let $B(v, t) = \mathbb{E} \sum_{x} b(x)$ where the sum goes over all particles at $v$ at time $t$ in the CTP, and let $S(v, t)$ denote the expected number of particles $x$ at $v$ at time $t$ such that $a(x)=0$.
 The probability that there is an uncontested particle at $v$ at time $t$ is at least $S(v, t) \exp\left(- \frac{B(v, t)}{S(v, t)}\right)$.
\end{theorem}
\begin{proof}
Let $P(v, t)$ denote the probability that $v$ contains an uncontested particle at time $t$. As at most one particle at each vertex can be uncontested, this is the same thing as the expected number of uncontested particles at $v$ at time $t$. For each path $\sigma$ from $\vz$ to $v$, let $P_\sigma(v, t)$, $B_\sigma(v, t)$ and $S_\sigma(v, t)$ denote the contribution to $P(v, t)$, $B(v, t)$ and $S(v, t)$ respectively from particles whose ancestral line follows $\sigma$. 

The idea now is to bound $P_\sigma(v, t)$ in terms of $B_\sigma(v, t)$ and $S_\sigma(v, t)$ for each path $\sigma$ from $\vz$ to $v$. Recall that $a(x)$ is constant over all particles $x$ whose ancestral line follows a fixed $\sigma$. We will denote this constant by $a(\sigma)$.

Let $\sigma$ be a path from $\vz$ to $v$ such that $a(\sigma)=0$. Applying Lemma \ref{lemma:palmCTP} we see that
\begin{equation}
S_\sigma(v, t) = \mathbb{E} \sum_{x\in V_\sigma(\mathbf{X})} \mathbbm{1}_{T(x) \leq t} = \int_0^\infty \dots \int_0^\infty \mathbbm{1}_{z_1+\dots+z_{l-1} \leq t}\, dz_1\dots\,dz_{l-1},
\end{equation}
where $T(x)$ denotes the time of birth of $x$. Similarly, for any $x\in V_\sigma(\mathbf{X})$ we have
\begin{equation}
\begin{split}
&B_\sigma(v, t) = \mathbb{E} \sum_{x\in V_\sigma(\mathbf{X})} \mathbbm{1}_{T(x)\leq t}\, b(\mathbf{X}, x)\\
&= \int_0^\infty \dots \int_0^\infty \mathbbm{1}_{z_1+\dots+z_{l-1} \leq t}\, \mathbb{E} b\left(\mathbf{X}^{z_1,\dots,z_{l-1}}, x^{z_1,\dots,z_{l-1}}\right)\,dz_1\dots\,dz_{l-1}
\end{split}
\end{equation}
and
\begin{equation}
\begin{split}
&P_\sigma(v, t) = \mathbb{E} \sum_{x\in V_\sigma(\mathbf{X})} \mathbbm{1}_{T(x)\leq t} \mathbbm{1}_{b(\mathbf{X}, x)=0}\\
&= \int_0^\infty \dots \int_0^\infty \mathbbm{1}_{z_1+\dots+z_{l-1} \leq t}\,\mathbb{P}\left( b\left(\mathbf{X}^{z_1,\dots,z_{l-1}}, x^{z_1,\dots,z_{l-1}}\right)=0\right)\,dz_1\dots\,dz_{l-1}.
\end{split}
\end{equation}
As $a(x^{z_1,\dots,z_{l-1}})=0$, no two ancestors of $x^{z_1,\dots,z_{l-1}}$ occupy the same vertex. It follows that any pair of particles $y$ and $z$ which is counted in $b\left(\mathbf{X}^{z_1,\dots,z_{l-1}}, x^{z_1,\dots,z_{l-1}}\right)$ is uniquely determined by $z$. Fixing $\sigma$ and $z_1, \dots, z_{l-1}$, this means that we can define $\phi(x)$ as an indicator function such that
\begin{equation}
b\left(\mathbf{X}^{z_1,\dots,z_{l-1}}, x^{z_1,\dots,z_{l-1}}\right) = \sum_{x\in \mathbf{X}^{z_1,\dots,z_{l-1}}} \phi(x).
\end{equation}
More precisely, $\phi(x)$ is the indicator for $x$ occupying the same vertex as an ancestor of $x^{z_1,\dots,z_{l-1}}$ and being born before it. By the definition of $b(x^{z_1,\dots,z_{l-1}})$, we have that $\phi(x)$ is zero for any ancestor or identical $n$-tuplet of an ancestor of $x^{z_1,\dots,z_{l-1}}$. It follows by Lemma \ref{lemma:poissonCTP} that
\begin{equation}
\mathbb{P}\left( b\left(\mathbf{X}^{z_1,\dots,z_{l-1}}, x^{z_1,\dots,z_{l-1}}\right)=0 \right) \geq \exp\left(-\mathbb{E} b\left(\mathbf{X}^{z_1,\dots,z_{l-1}}, x^{z_1,\dots,z_{l-1}}\right)\right).
\end{equation}

By convexity of the exponential function we have $e^{-r} \geq (1+r_0-r)e^{-r_0}$ for any $r, r_0\in\mathbb{R}$. Hence
\begin{equation}\label{eq:PsSsBs}
P_\sigma(v, t) \geq (1+r_0)e^{-r_0} S_\sigma(v, t) - e^{-r_0} B_\sigma(v, t),
\end{equation}
for any path $\sigma$ from $\vz$ to $v$ such that $a(\sigma)=0$. For any $\sigma$ that satisfies $a(\sigma)\neq 0$ it is clear that $P_\sigma(v, t)=S_\sigma(v, t)= 0$ and $B_\sigma(v, t) \geq 0$, hence \eqref{eq:PsSsBs} holds in this case as well. Summing over all paths $\sigma$ from $\vz$ to $v$ and optimizing over $r_0$ yields $P(v, t) \geq S(v, t) e^{-\frac{B(v, t)}{S(v, t)}}$, as desired.
\end{proof}

We will apply Theorem \ref{thm:uncontestedCTP} as follows: Let $\{\vv_n\}_{n=1}^\infty$ be a sequence of vertices such that, for each $n$, $\vv_n \in \hc{n}$ and $x=\lim_{n\rightarrow\infty} \abs{\vv_n}/n$ exists and is non-zero. We may, without loss of generality, assume that $\vv_n$ is never equal to $\vz$. For each $n$, we let $\vartheta_n$ denote the unique non-negative solution to
\begin{equation}\label{eq:defvarthetan}
m(\vv_n, \vartheta_n) = \frac{1}{n}.
\end{equation}
Note that the expected number of particles at $\vv_n$ at time $\vartheta_n$ in a standard CTP on $\hc{n}$ is $\Theta(1)$, and that $\vartheta_n\rightarrow \vartheta(x)$ as $n\rightarrow\infty$. By Theorem \ref{thm:uncontestedCTP} we have that the probability that there is a uncontested particle at $\vv_n$ at time $\vartheta_n$ in the CTP on $\hc{n}$ is at least $S(\vv_n, \vartheta_n) \exp\left(-\frac{B(\vv_n, \vartheta_n)}{S(\vv_n, \vartheta_n)}\right)$. Hence by Lemma \ref{lemma:uncontestedisalive} and Proposition \ref{prop:fppisaliveCTP} it follows that
\begin{equation}
\mathbb{P}\left(\mathcal{T}_V'(\vz, \vv_n) \leq \vartheta_n\right) \geq S(\vv_n, \vartheta_n) \exp\left(-\frac{B(\vv_n, \vartheta_n)}{S(\vv_n, \vartheta_n)}\right).
\end{equation}
This means that if we can show that $S(\vv_n, \vartheta_n)=\Theta(1)$ and $B(\vv_n, \vartheta_n)=O(1)$, then we know that $\mathcal{T}_V'(\vz, \vv_n) \leq \vartheta_n$ with probability bounded away from $0$ as $n\rightarrow\infty$.

Section \ref{sec:calculus} will be dedicated to estimating $S(\vv_n, \vartheta_n)$ and $B(\vv_n, \vartheta_n)$. The proof of Theorem \ref{thm:mainresult1equiv} is then completed in Section \ref{sec:bootstrap} by showing that if $\mathcal{T}_V'(\vz, \vv_n) \leq \vartheta_n$ with probability bounded away from $0$, then a slightly larger upper bound on $\mathcal{T}_V'(\vz, \vv_n)$ must hold asymptotically almost surely.

\section{Calculus}\label{sec:calculus}

\subsection{Estimating $S$}
We will prove that $S(\vv_n, \vartheta_n) = \Theta(1)$ in two steps. Firstly, we show that most particles at $\vv_n$ at time $\vartheta_n$ have ancestral lines which are close to vertex-minimal. Using this, we then give a combinatorial argument that shows that a positive proportion of these particles must have vertex-minimal ancestral lines.

Let us formalize the notion of paths being close to vertex-minimal. Let $v, w\in \hc{n}$ be fixed distinct vertices and let $\sigma=\{v=v_0, v_1, \dots, v_l=w\}$ be a path from $v$ to $w$. Throughout this section, we will always think of a path as a finite sequence of vertices. In particular, by the length of a path we mean the number of vertices in the path. For any $0 < i \leq j < l$ we say that the subsequence $v_i, v_{i+1}, \dots, v_j$ is a \emph{detour} of $\sigma$ if removing these elements from $\sigma$ results in a valid path. Clearly, for $v\neq w$ a path is vertex-minimal if and only if it has no detours. Inspired by this, we say that a path is \emph{almost vertex-minimal} if all detours have length at most $2$. Note that as $\hc{n}$ is bipartite, any detour must have even length. Hence, a path is almost vertex-minimal if it only has the shortest possible detours.

An important property of almost vertex-minimal paths is that any such path from $v$ to $w$ can be constructed by taking a vertex-minimal path with the same end-points and extending it as follows: Between each two adjacent elements in the sequence either do nothing or insert a detour of length $2$.

\begin{lemma}\label{lemma:mconv}
Let $s, t\geq 0$ and $v\in \mathbb{Q}_n$. Then
\begin{equation}
\sum_{w\in \mathbb{Q}_n} m(w, s)m(v+w, t) = m(v, s+t).
\end{equation}
\end{lemma}
\begin{proof}
Fix $s$. Observe that equality holds when $t=0$ and that both expressions solves \eqref{eq:MAdeq}
\end{proof}

\begin{proposition}\label{prop:vertexminimal1}
Let $\{\vv_n\}_{n=1}^\infty$ be a sequence of vertices, $\vv_n \in \hc{n}$, such that $\alpha = \lim_{n\rightarrow\infty} \abs{\vv_n}/n$ exists and is positive. Then, as $n\rightarrow \infty$, the expected number of particles in the standard CTP on $\hc{n}$ which are at $\vv_n$ at time $\vartheta_n$, but that do not have almost vertex-minimal ancestral lines tends to $0$.
\end{proposition}
\begin{proof}
Let $X_n$ denote the number of triples of particles $x, y, z$ in the CTP on $\hc{n}$ such that
\begin{itemize}
\item $x$ is at $\vv_n$ at time $\vartheta_n$
\item $y$ and $z$ are located at adjacent vertices
\item $z$ is an ancestor of $y$ which is an ancestor of $x$.
\item $y$ and $z$ are neither one nor three generations apart.
\end{itemize}
We note that if the ancestral line of a given particle $x$ at $\vv_n$ at time $\vartheta_n$ can be constructed using some detour of length $d>2$, then it is clear that $x$ would have a pair of ancestors at adjacent vertices which are $d+1$ generations apart. This means that any such $x$ is counted at least once in $X_n$. Hence, it suffices to show that $\mathbb{E} X_n = o(1)$.

For each triple $x, y, z$ as above there are uniquely defined particles $c$, the particle after $z$ in the ancestral line of $x$, and $p$, the parent of $y$. Note that the requirement that $y$ is neither the child, nor the grand-grandchild of $z$ implies that $p$ is a descendant of $c$, but not a child of $c$.

Let $T=\{0 = t_0 < t_1 < \dots < t_k = \vartheta_n\}$ denote the end-points of a partition of $[0, \vartheta_n)$ into left-closed right-open subintervals, and let $X_{n,T}$ denote the number of triples as above where $c$ and $y$ are the only ancestors of $x$ born during their respective time intervals. Pick $a, b$ integers between $0$ and $k-1$. Consider the number of triples counted in $X_{n,T}$ where $c$ is born during $[t_a, t_{a+1})$ and $y$ is born during $[t_b, t_{b+1})$. Note that this is trivially $0$ whenever $b \leq a$.

Let us count the expected number of corresponding triples for $a < b$. As $z$ and $y$ are located at adjacent vertices, for each such triple we may denote the locations of $z$, $y$, $c$ and $p$ by $v$, $v+e_i$, $v+e_j$ and $v+e_i-e_k$ respectively for some $v\in \hc{n}$ and $1 \leq i, j, k \leq n$. A particle is a potential $z$ if it is born before time $t_a$, hence there are on average $\sum_{l=1}^n m(v-e_l, t_a)$ potential $z$:s at $v$. For each $z$, a particle is a potential $c$ if it is a child of $z$ born during $[t_a, t_{a+1})$. Hence for each potential $z$ at $v$, there are on average $t_{a+1}-t_a$ potential $c$:s at $v+e_j$. For each potential $c$, a particle is a potential $p$ if it originates from $c$ at time $t_{a+1}$ and is born before $t_b$, but is not a child of $c$. Hence for each potential $c$ at $v+e_j$ there are on average $m(e_i-e_k-e_j, t_{b}-t_{a+1})$ potential $p$:s at $v+e_i-e_k$ if $v+e_j$ and $v+e_i-e_k$ are not adjacent, and $m(e_i-e_k-e_j, t_{b}-t_{a+1}) - (t_{b}-t_{a+1})$ if they are. Lastly, for each potential $p$, a particle is a potential $y$ if it is a child of $p$ born during $[t_b, t_{b+1})$, and for each potential $y$ a particle is a potential $x$ if it is located at $\vv_n$, originates from $y$ at time $t_{b+1}$, and is born before time $\vartheta_n$. Hence for each potential $p$ at $v+e_i-e_k$ the expected number of potential $y$:s at $v+e_i$ is $t_{b+1}-t_b$, and for each potential $y$ at $v+e_i$, the expected number of $x$:s is $m(\vv_n-v-e_i, \vartheta_n-t_{b+1})$. Combining all of these, we see that
\begin{equation}\label{eq:riemannS}
\begin{split}
\mathbb{E} X_{n,T} &= \sum_{a < b} \sum_{v\in \hc{n}} \sum_{i, j, k, l} m(v-e_l, t_a) (t_{a+1}-t_a)  \biggl( m(e_i-e_k-e_j, t_{b}-t_{a+1})\\
&\qquad - \mathbbm{1}_{\abs{e_i+e_j+e_k}=1}(t_{b}-t_{a+1})\biggr) (t_{b+1}-t_b) m(\vv_n-v-e_i, \vartheta_n-t_{b+1}).
\end{split}
\end{equation}
where the sums over $i$, $j$, $k$ and $l$ all go from $1$ to $n$. Letting $T_1, T_2, \dots$ be a sequence of increasingly finer partitions of $[0, \vartheta_n]$ such that the length of the longest interval in $T_k$ tends to $0$ as $k\rightarrow\infty$, it follows by monotone convergence that we have $\mathbb{E} X_n = \lim_{k\rightarrow\infty} \mathbb{E} X_{n,T}$. Combining this with equation \eqref{eq:riemannS}, and recognizing the right-hand side as a Riemann sum, we get
\begin{equation}\label{eq:nonalmostvmin}
\begin{split}
\mathbb{E} X_n &= \int_0^{\vartheta_n} \int_a^{\vartheta_n} \sum_{v\in \hc{n}} \sum_{i, j, k, l} m(v-e_l, a) \biggl( m(e_i-e_k-e_j, b-a)\\
&\qquad  - \mathbbm{1}_{\abs{e_i+e_j+e_k}=1}(b-a)\biggr) m(\vv_n-v-e_i, \vartheta_n-b)\,da\,db.
\end{split}
\end{equation}

Lemma \ref{lemma:mconv} implies that we may replace the factor $\sum_{v\in \hc{n}} m(v-e_l, a) m\left(\vv_n-v-e_i, \vartheta_n-b\right)$ in the integrand of equation \eqref{eq:nonalmostvmin} by $m(\vv_n+e_i+e_l, \vartheta_n-b+a)$. Hence, by the substitution $t=b-a$, the right-hand side of \eqref{eq:nonalmostvmin} simplifies to
\begin{equation}
\begin{split}
&\int_0^{\vartheta_n} (\vartheta_n-t) \sum_{i,j,k,l} m(\vv_n+e_i+e_l, \vartheta_n-t)\cdot \\
&\qquad \cdot\left( m(e_i+e_j+e_k, t) - \mathbbm{1}_{\abs{e_i+e_j+e_k}=1} t \right)\,dt.
\end{split}
\end{equation}
Using the fact that $\sinh t \leq \cosh t$ for all $t \in \mathbb{R}$, we have
\begin{align*}
&\sum_{i,j,k,l} m(\vv_n+e_i+e_l, \vartheta_n-t) \left( m(e_i+e_j+e_k, t) - \mathbbm{1}_{\abs{e_i+e_j+e_k}=1} t \right)\\
&\leq 
n\, \left( \sinh (\vartheta_n-t)\right)^{\abs{\vvs_n}-2} \left(\cosh(\vartheta_n-t)\right)^{n-\abs{\vvs_n}+2} \sum_{i,j,k} \left( m(e_i+e_j+e_k, t) - \mathbbm{1}_{\abs{e_i+e_j+e_k}=1} t \right).
\end{align*}
It is straight-forward (but messy) to show that $\sum_{i,j,k} \left( m(e_i+e_j+e_k, t) - \mathbbm{1}_{\abs{e_i+e_j+e_k}=1} t \right) = \left(\cosh t\right)^n\,O\left(n^3 t^3\right)$. As $\cosh\left( \vartheta_n-t \right) \cosh t \leq \cosh \vartheta_n$ it follows that
\begin{equation}\label{eq:almostvmin2}
\mathbb{E} X_n \leq \int_0^{\vartheta_n} n \left( \sinh(\vartheta_n-t) \cosh t \right)^{\abs{\vvs_n}-2} \left(\cosh \vartheta_n\right)^{n-\abs{\vvs_n}+2}\,O\left( n^3 t^3\right) \, dt.
\end{equation}

Recall that by the definition of $\vartheta_n$ we have
\begin{equation}
\left(\sinh \vartheta_n\right)^{\abs{\vvs_n}}\left(\cosh \vartheta_n\right)^{n-\abs{\vvs_n}}=\frac{1}{n}.
\end{equation}
Define the function $f(t) = \ln \sinh(\vartheta_n-t) + \ln \cosh t$. Note that $f'(t) = -\coth(\vartheta_n-t)+\tanh t$, and $f''(t) = -\csch^2(\vartheta_n-t) + \sech^2 t$. As $0 \leq \sech t \leq 1$ and $\csch t \geq 1$ for all $0<t\leq \vartheta_n$ it follows that $f$ is concave, and thus for any $0 \leq t \leq \vartheta_n$ we have $f(t) \leq f(0) - t\,\coth \vartheta_n \leq f(0)-t$. Hence
\begin{equation}
\left( \sinh(\vartheta_n-t) \cosh t \right)^{\abs{\vvs_n}-2} \leq \left(\sinh \vartheta_n\right)^{\abs{\vvs_n}-2} e^{-(\abs{\vvs_n}-2)t}.
\end{equation}
Plugging this into equation \eqref{eq:almostvmin2}, we get
\begin{equation}
\mathbb{E} X_n \leq \int_0^{\vartheta_n} e^{-(\abs{\vvs_n}-2)t }\,O\left( n^3 t^3 \right) \, dt.
\end{equation}
As $\abs{\vv_n} \sim x\cdot n$ this implies that $\mathbb{E}X_n = O\left(\frac{1}{n}\right)$, as desired.
\end{proof}

\begin{proposition}\label{prop:boundS}
For any pair of sequences $\{\vv_n\}_{n=1}^\infty$ and $\{\vartheta_n\}_{n=1}^\infty$ as above, we have $S(\vv_n, \vartheta_n) = \Theta(1)$.
\end{proposition}
\begin{proof}
Let $\Gamma_n$ and $\tilde{\Gamma}_n$ denote the sets of vertex-minimal and almost vertex-minimal paths from $\vz$ to $\vv_n$ respectively. Using Lemma \ref{lemma:palmCTP} with $f(X, x)$ as the indicator function of $x$ being born at time $\vartheta_n$ and having ancestral line in $\tilde{\Gamma}_n$ and $\Gamma_n$ respectively, we can write the expected number particles at $\vv_n$ at time $\vartheta_n$ in the CTP whose ancestral lines are almost vertex-minimal as
\begin{equation}\label{eq:gammatilde}
\sum_{\sigma \in \tilde{\Gamma}_n} \frac{ \vartheta_n^{\abs{\sigma}-2}}{(\abs{\sigma}-2)!}
\end{equation}
and the expected number that are vertex-minimal as
\begin{equation}\label{eq:gammanotilde}
\sum_{\sigma \in \Gamma_n} \frac{ \vartheta_n^{\abs{\sigma}-2}}{(\abs{\sigma}-2)!}.
\end{equation}
As the total expected number of particles at $\vv_n$ at time $\vartheta_n$ in the CTP is $\Theta(1)$, Proposition \ref{prop:vertexminimal1} implies that the sum in \eqref{eq:gammatilde} is also $\Theta(1)$.

The idea now is to group the terms of the sum in \eqref{eq:gammatilde} according to which vertex-minimal path $\sigma$ it is an extension of, that is
we write
\begin{equation}
\sum_{\sigma \in \tilde{\Gamma}_n} \frac{ \vartheta_n^{\abs{\sigma}-2}}{(\abs{\sigma}-2)!} \leq \sum_{\sigma\in\Gamma_n}\sum_{\substack{\tilde{\sigma} \in \tilde{\Gamma}_n\\ \tilde{\sigma} \supseteq \sigma}} \frac{ \vartheta_n^{\abs{\tilde{\sigma}}-2}}{(\abs{\tilde{\sigma}}-2)!}.
\end{equation}
Here $\tilde{\sigma} \supseteq \sigma$ denotes that $\tilde{\sigma}$ is an extension of $\sigma$. Note that the inequality comes from the fact that $\tilde{\sigma}$ may be an extension of more than one vertex-minimal path.

Let us fix a vertex-minimal path $\sigma \in \Gamma_n$ consisting of $l$ vertices. It is straight-forward to show that the number of possible detours of length $2$ that can be inserted between each adjacent pair of elements in $\sigma$ is $3(n-1)$. Hence, there are at most $3^k(n-1)^k {l-1 \choose k}$ ways to extend $\sigma$ to an almost vertex-minimal path of length $l+2k$. This means that
\begin{align*}
\sum_{\substack{\tilde{\sigma} \in \tilde{\Gamma}_n \\ \tilde{\sigma} \supseteq \sigma}} \frac{ \vartheta_n^{\abs{\tilde{\sigma}}-2}}{(\abs{\tilde{\sigma}}-2)!} &\leq \sum_{k=0}^{l-1} 3^k (n-1)^k {l-1 \choose k}
\frac{\vartheta_n^{l-2+2k}}{ (l-2+2k)!}\\
&\leq \frac{\vartheta_n^{l-2}}{(l-2)!} \sum_{k=0}^{l-1} 3^k (n-1)^k {l-1 \choose k}\frac{\vartheta_n^{2k}}{ (l-1)^{2k}}\\
& = \frac{\vartheta_n^{l-2}}{(l-2)!} \left(1+\frac{3 \vartheta_n^2(n-1)}{(l-1)^2} \right)^{l-1}\\
& \leq \frac{\vartheta_n^{l-2}}{(l-2)!} \exp\left(\frac{3\vartheta_n^2(n-1)}{l-1}\right).
\end{align*}
As any path from $\vz$ to $\vv_n$ must have length at least $\abs{\vv_n}+1$, we conclude that
\begin{equation}
\sum_{\sigma \in \Gamma_n } \frac{\vartheta_n^{\abs{\sigma}-2} }{(\abs{\sigma}-2)!} \geq \exp\left( - 3 \vartheta_n^2 \frac{n-1}{\abs{\vv_n}}\right) \sum_{\sigma \in \tilde{\Gamma}_n } \frac{ \vartheta_n^{\abs{\sigma}-2}}{(\abs{\sigma}-2)!}  = \Theta(1).
\end{equation}
\end{proof}

\subsection{Estimating $B$}

\begin{proposition}\label{prop:fugly}
For any $\vv \in \hc{n}$ and any $u > 0$ we have
\begin{equation}\label{eq:fugly}
\begin{split}
&B(\vv, u) \leq \int_0^u \sum_{\Delta \in \hc{n}}\sum_{i, j, k} m(\Delta-e_k-e_i, t)\, m(\Delta-e_j, t)\, m(\vv-\Delta, u-t) \, dt\\
&\qquad +\int_0^u (u-t) \sum_{\Delta \in \hc{n}} \sum_{i, j, k}  m(\Delta-e_k-e_j, t)\, m(\Delta, t) m(\vv-\Delta-e_i, u-t)\,dt\\
&\qquad + \int_0^u (u-t) \sum_{\Delta \in \hc{n}} \sum_{i, j, k} m(\Delta-e_k, t)\, m(\Delta-e_j, t) m(\vv-\Delta-e_i, u-t) \,dt\\
&\qquad +\int_0^t (u-t) \sum_{\Delta \in \hc{n}} \sum_{i, j, k, l} m(\Delta-e_l-e_j, t)\, m(\Delta-e_k, t) m(\vv-\Delta-e_i, u-t)\,dt,
\end{split}
\end{equation}
where the sums over $i$, $j$, $k$ and $l$ go from $1$ to $n$.
\end{proposition}
\begin{proof}
We observe that $B(\vv, u)$ is bounded by the expected number of triplets of particles $x, y, z$ in the CTP such that
\begin{itemize}
\item $x$ is at $\vv$ at time $u$
\item $y$ is an ancestor of $x$
\item $y$ and $z$ occupy the same vertex
\item $z$ was born before $y$.
\end{itemize}
Note the similarity to the quantity $X_n$ in Proposition \ref{prop:vertexminimal1}. For the sake of compactness, we will be less rigorous here, and refer to the proof of that proposition to see how to formalize this argument.

Let us start by considering the number of such triples $x$, $y$ and $z$ where $z$ has no ancestors in common with $x$ and $y$, that is, for some $i \neq j$ we have that $x$ and $y$ originate from the original particle at $e_i$ whereas $z$ originates from the original particle at $e_j$. Denote the common location of $y$ and $z$ by $v$, and pick $k$ such that the parent of $y$ is located at $v-e_k$. Note that as $z$ is strictly older than $y$, $y$ cannot be an original particle and hence has a parent. The lineage of $x, y, z$ is illustrated in Graph 1 of Figure \ref{fig:heredity}.

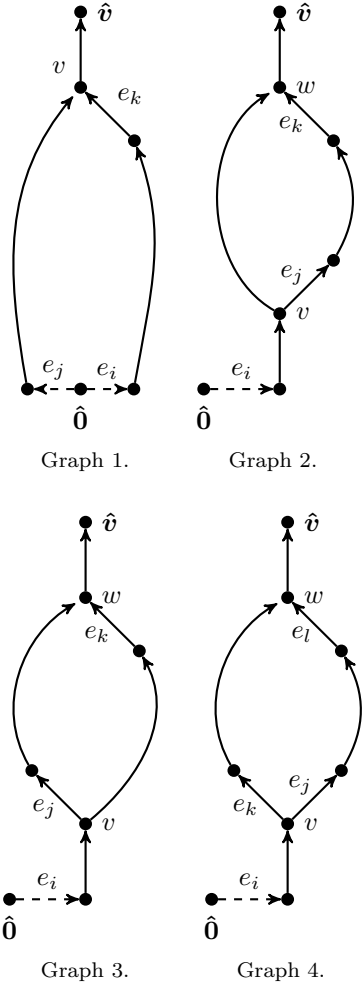
\begin{figure}[ht]
\captionsetup[subfigure]{labelformat=empty}

\begin{center}
\subfloat[][Graph 1.]{
\begin{tikzpicture}[->,>=stealth',shorten >=0pt,auto,node distance=3cm,thick,main node/.style={circle,fill=black,draw,minimum size=4pt,inner sep=0pt}]

\node[main node] (vz) [label=below:$\vz$] {};
\node[main node] (xanc) [right of=vz] [node distance=0.7cm] {};
\node[main node] (zanc) [left of=vz] [node distance=0.7cm] {};
\node[main node] (v) [above of=vz] [label=above left:$v$] [node distance=4cm] {};
\node[main node] (p) [below right of=v] [node distance=1.0cm] {};
\node[main node] (vv) [above of=v] [label=right:$\vv$] [node distance=1cm] {};

\path
(vz) edge [dashed] node[above] {$e_i$} (xanc)
(vz) edge [dashed] node[above] {$e_j$} (zanc)
(p) edge node[above right] {$e_k$} (v)
(xanc) edge [out=80,in=-70,looseness=1] (p)
(zanc) edge [out=100,in=230,looseness=1,shorten >=2pt] (v)
(v) edge (vv);
\end{tikzpicture}
}
\subfloat[][Graph 2.]{
\begin{tikzpicture}[->,>=stealth',shorten >=0pt,auto,node distance=3cm,thick,main node/.style={circle,fill=black,draw,minimum size=4pt,inner sep=0pt}]

\node[main node] (vz) [label=below:$\vz$] {};
\node[main node] (anc) [right of=vz] [node distance=1cm] {};
\node[main node] (v) [above of=anc] [label=right:$v$] [node distance=1cm] {};
\node[main node] (w) [above of=v] [label=right:$w$] [node distance=3cm] {};
\node[main node] (c) [above right of=v] [node distance=1.0cm] {};
\node[main node] (p) [below right of=w] [node distance=1.0cm] {};
\node[main node] (vv) [above of=w] [label=right:$\vv$] [node distance=1cm] {};

\path
(vz) edge [dashed] node[above] {$e_i$} (anc)
(anc) edge (v)
(v) edge node[above left=-0.1cm] {$e_j$} (c)
(c) edge [out=60,in=-60,looseness=1] (p)
(p) edge node[below left=-0.1cm] {$e_k$} (w)
(v) edge [out=150,in=210,looseness=1,shorten >=2pt] (w)
(w) edge (vv);

\end{tikzpicture}
}

\subfloat[][Graph 3.]{
\begin{tikzpicture}[->,>=stealth',shorten >=0pt,auto,node distance=3cm,thick,main node/.style={circle,fill=black,draw,minimum size=4pt,inner sep=0pt}]

\node[main node] (vz) [label=below:$\vz$] {};
\node[main node] (anc) [right of=vz] [node distance=1cm] {};
\node[main node] (v) [above of=anc] [label=right:$v$] [node distance=1cm] {};
\node[main node] (w) [above of=v] [label=right:$w$] [node distance=3cm] {};
\node[main node] (c) [above left of=v] [node distance=1.0cm] {};
\node[main node] (p) [below right of=w] [node distance=1.0cm] {};
\node[main node] (vv) [above of=w] [label=right:$\vv$] [node distance=1cm] {};

\path
(vz) edge [dashed] node[above] {$e_i$} (anc)
(anc) edge (v)
(v) edge node[below left=-0.1cm] {$e_j$} (c)
(v) edge [out=40,in=-60,looseness=1] (p)
(p) edge node[below left=-0.1cm] {$e_k$} (w)
(c) edge [out=120,in=210,looseness=1,shorten >=2pt] (w)
(w) edge (vv);

\end{tikzpicture}
}
\subfloat[][Graph 4.]{
\begin{tikzpicture}[->,>=stealth',shorten >=0pt,auto,node distance=3cm,thick,main node/.style={circle,fill=black,draw,minimum size=4pt,inner sep=0pt}]

\node[main node] (vz) [label=below:$\vz$] {};
\node[main node] (anc) [right of=vz] [node distance=1cm] {};
\node[main node] (v) [above of=anc] [label=right:$v$] [node distance=1cm] {};
\node[main node] (w) [above of=v] [label=right:$w$] [node distance=3cm] {};
\node[main node] (c1) [above left of=v] [node distance=1.0cm] {};
\node[main node] (c2) [above right of=v] [node distance=1.0cm] {};
\node[main node] (p) [below right of=w] [node distance=1.0cm] {};
\node[main node] (vv) [above of=w] [label=right:$\vv$] [node distance=1cm] {};

\path
(vz) edge [dashed] node[above] {$e_i$} (anc)
(anc) edge (v)
(v) edge node[below left=-0.1cm] {$e_k$} (c1)
(v) edge node[above left=-0.1cm] {$e_j$} (c2)
(c2) edge [out=60,in=-60,looseness=1] (p)
(p) edge node[below left=-0.1cm] {$e_l$} (w)
(c1) edge [out=120,in=210,looseness=1,shorten >=2pt] (w)
(w) edge (vv);

\end{tikzpicture}
}
\caption{Illustration of the possible ways $x$, $y$ and $z$ can be related. The left-most arrows describe the ancestors of $z$ and the right-most the ancestors of $x$ and $y$. Graph 1 shows the case when $z$ has no ancestor in common with $x$ and $y$. Here $v$ is the common location of $y$ and $z$, and $v-e_k$ is the location of the parent of $y$. For Graphs 2-4, $v$ denotes the location of the last common ancestor of $x$ and $z$, and $w$ the common location of $y$ and $z$. Graph 2 shows the case where the ancestral lines of $x$ and $z$ split by the birth of a new ancestor of $x$, Graph 3 the case where this occurs by a new ancestor of $z$ and Graph 4 the case where the first unique ancestors of $x$ and of $z$ are born simultaneously as part of the same group of identical $n$-tuplets.
}\label{fig:heredity}
\end{center}
\end{figure}

Let us count the expected number of such triples corresponding to a fixed $v$ and where $y$ is born during the time interval $[t, t+dt)$. The potential $z$:s corresponding to a fixed $j$ are simply the descendants of the original particle at $e_j$ that are at $v$ at time $t$. Hence the expected number of such particles is $m(v-e_j, t)$. Similarly, for a fixed $i$ the expected number of potential $y$:s is given by $m(v-e_k-e_i, t)\,dt$, and for each potential $y$ the expected number of potential $x$:s is $m(v, u-t)$. As the potential $z$:s are born independently of the pairs of potential $x$:s and $y$:s, we see that the expected number of triples $x, y, z$ that do not have common ancestors, corresponding to a fixed vertex $v$ and a fixed time interval $[t, t+dt)$ is given by
\begin{equation}\label{eq:G4}
\sum_{i=1}^n \sum_{\substack{j=1\\j\neq i}}^n \sum_{k=1}^n m(v-e_k-e_i, t)\, m(v-e_j, t)\, m(\vv-v, u-t) \, dt.
\end{equation}
The total expected number of triples $x, y, z$ without common ancestors is hence given by summing this expression over all vertices $v\in \hc{n}$ and integrating over $t$ from $0$ to $u$. This is clearly bounded from above by the first term in the right-hand side of equation \eqref{eq:fugly}.

We now consider the cases where the three particles $x, y, z$ have common ancestors. Denote the last common ancestor of the particles by $l$ and its location by $v$. As $x$ and $z$ have common ancestors but neither is a descendant of the other, there must be a time $s$ when the ancestral lines of $x$ and $z$ split. There are three possible ways in which this can occur, as illustrated by Graphs 2-4 in Figure \ref{fig:heredity}; either a new ancestor of $x$ is born, a new ancestor of $z$ is born, or new ancestors of $x$ and $z$ are identical $n$-tuplets and therefore born at the same time. Observe that, in all three cases, $y$ must be born strictly after this time. We let $w$ denote the common location of $y$ and $z$.

We now count the expected number of such triples corresponding to fixed vertices $v$ and $w$, where the ancestral lines split during the time interval $[s, s+dt)$ and such that $y$ is born during $[s+t, s+t+dt)$. The potential $l$:s are the particles in the CTP at $v$ at time $s$, hence the expected number of potential $l$:s is $\sum_{i=1}^n m(v-e_i, s)$. For each potential $l$, the probability that it gives birth during $[s, s+ds)$ is $ds$. Now, for each possibility for the ancestral lines of $x$ and $z$ to split, conditioned on the process at time $s+ds$, the pairs of potential $x$:s and $y$:s originate from a different particle than the potential $z$:s. Hence these are born independently. By following the ancestral lines as illustrated in Graphs 2-4 in a similar manner as above, we see that the expected number of triples with common ancestors corresponding to fixed $v$ and $w$, fixed time intervals, and corresponding to each case for how the ancestral line splits are given by
\begin{equation}\label{eq:G1}
\sum_{i=1}^n \sum_{j=1}^n \sum_{k=1}^n m(v-e_i, s)\, m(w-e_k-v-e_j, t)\, m(w-v, t)\, m(\vv-w, u-s-t)\,ds\,dt
\end{equation}
\begin{equation}\label{eq:G2}
\sum_{i=1}^n \sum_{j=1}^n \sum_{k=1}^n m(v-e_i, s)\, m(w-e_k-v, t)\, m(w-v-e_j, t)\, m(\vv-w, u-s-t)\,ds\,dt
\end{equation}
\begin{equation}\label{eq:G3}
\sum_{i=1}^n \sum_{j=1}^n \sum_{\substack{k=1\\ k\neq j}}^n \sum_{l=1}^n m(v-e_i, s)\, m(w-e_l-v-e_j, t)\, m(w-v-e_k, t)\, m(\vv-w, u-s-t)\,ds\,dt
\end{equation}
respectively. The total expected number of triples $x, y, z$ with common ancestors is hence given by summing these three expressions over all pairs of vertices $v, w\in \hc{n}$ and integrating over all $s$ and $t$ such that $s, t \geq 0$ and $s+t \leq u$.

It only remains to simplify these expressions. We observe that summing \eqref{eq:G1}, \eqref{eq:G2} and \eqref{eq:G3} over all $v, w\in \hc{n}$ removes all dependence on $s$. Consider in particular the sum of \eqref{eq:G1} over all $v, w\in \hc{n}$. By substituting summing over $w$ by summing over $\Delta=w-v$ and applying Lemma \ref{lemma:mconv} we have
\begin{equation}
\begin{split}
&\sum_{v, \Delta \in \hc{n}} \sum_{i, j, k} m(v-e_i, s)\, m(\Delta-e_k-e_j, t)\, m(\Delta, t)\, m(\vv-\Delta-v, u-s-t)\,ds\,dt\\
&\qquad = \sum_{\Delta \in \hc{n}} \sum_{i, j, k}  m(\Delta-e_k-e_j, t)\, m(\Delta, t) m(\vv-\Delta-e_i, u-t)\,ds\,dt.
\end{split}
\end{equation}
Integrating this expression over all $s, t\geq 0$ such that $s+t\leq u$, we see that the expected number of triples of particles $x, y, z$ as above corresponding to the case illustrated in Graph 2 in Figure \ref{fig:heredity} is given by
\begin{equation}\label{eq:fuglyterm2}
\int_0^u (u-t) \sum_{\Delta \in \hc{n}} \sum_{i, j, k}  m(\Delta-e_k-e_j, t)\, m(\Delta, t) m(\vv-\Delta-e_i, u-t)\,dt.
\end{equation}
Proceeding analogously for \eqref{eq:G2} and \eqref{eq:G3} we see that the expected number of triples corresponding to Graphs 3 and 4 in Figure \ref{fig:heredity} are given respectively by
\begin{equation}\label{eq:fuglyterm3}
\int_0^u (u-t) \sum_{\Delta \in \hc{n}} \sum_{i, j, k} m(\Delta-e_k, t)\, m(\Delta-e_j, t) m(\vv-\Delta-e_i, u-t) \,dt
\end{equation}
and
\begin{equation}\label{eq:fuglyterm4}
\int_0^t (u-t) \sum_{\Delta \in \hc{n}} \sum_{\substack{i, j, k, l\\ j\neq k}} m(\Delta-e_l-e_j, t)\, m(\Delta-e_k, t) m(\vv-\Delta-e_i, u-t)\,dt.
\end{equation}
The expressions in \eqref{eq:fuglyterm2}-\eqref{eq:fuglyterm4} are clearly bounded from above by terms 2-4 respectively in the right-hand side of equation \eqref{eq:fugly}.
\end{proof}

Consider the sum $\sum_{\Delta \in \hc{n}} m(\Delta, a)^2\,m(\vv-\Delta, b)$. For any $v\in \hc{n}$ we let $v^i$ denote the $i$:th coordinate of $v$. Define the function $m_1:\{0, 1\}\times \mathbb{R}\rightarrow \mathbb{R}$ by $m_1(0, t) = \cosh t$ and $m_1(1, t)=\sinh t$. Using the fact that $m(v, t) = \prod_{i=1}^n m_1(v^i, t)$, we see that
\begin{align*}
&\sum_{\Delta \in \hc{n}} m(\Delta, a)^2\,m(\vv-\Delta, b)\\
&\qquad = \sum_{\Delta \in \hc{n} } \prod_{i=1}^n m_1(\Delta^i, a)^2\,m_1(\vv^i+\Delta^i, b)\\
&\qquad = \prod_{i=1}^n \sum_{\delta=0}^1 m_1(\delta, a)^2\,m_1(\vv^i+\delta, b)\\
&\qquad = \left( \cosh(a)^2\,\sinh(b)+\sinh(a)^2\,\cosh(b)\right)^k \left( \sinh(a)^2\,\sinh(b)+\cosh(a)^2\,\cosh(b)\right)^{n-k}\\
&\qquad = e^{nb} \left( \frac{1}{2} \cosh 2a - \frac{1}{2} e^{-2b} \right)^k \left( \frac{1}{2} \cosh 2a + \frac{1}{2} e^{-2b} \right)^{n-k},
\end{align*}
where $k=\abs{\vv}$. Let
\begin{equation}
\begin{split}
G_x(a, b) &= x \ln\left( \cosh(a)^2\,\sinh(b)+\sinh(a)^2\,\cosh(b)\right)\\
&\qquad+ (1-x)\ln\left( \sinh(a)^2\,\sinh(b)+\cosh(a)^2\,\cosh(b)\right)\\
&= b + x \ln\left(\frac{1}{2} \cosh 2a - \frac{1}{2} e^{-2b}\right) + (1-x)\ln\left(\frac{1}{2} \cosh 2a + \frac{1}{2} e^{-2b}\right).
\end{split}
\end{equation}
Then
\begin{equation}\label{eq:nicesumab2}
\sum_{\Delta \in \hc{n}} m(\Delta, a)^2\,m(\vv-\Delta, b) = \exp\left(n G_{\frac{k}{n}}(a, b)\right).
\end{equation}

\begin{proposition}\label{prop:Bgettingnicer}
For any $\varepsilon>0$ there exists a constant $C_\varepsilon>0$ only depending on $\varepsilon$ such that whenever $u \in [\varepsilon, 1]$ we have
\begin{equation}\label{eq:Bgettingnicer}
B(\vv, u) \leq C_\varepsilon \int_0^u \left[\left(n^4 t^3+n^3t+n^2\right)(u-t)+\left(n^3 t^3+n^2t+n\right)\right]\exp\left(n G_{\frac{k}{n}}(t, u-t)\right)\,dt,
\end{equation}
where $k=\abs{\vv}$.
\end{proposition}
\begin{proof}
The idea of the proof is to use equation \eqref{eq:nicesumab2} to reformulate equation \eqref{eq:fugly} in terms of partial derivatives of $G_x(a, b)$. Note that by the fact that $m(v, t)$ satisfies \eqref{eq:MAdeq} we have
\begin{align*}
&\sum_{\Delta \in \hc{n}} \sum_{i, j, k}  m(\Delta-e_k-e_j, a)\, m(\Delta, a) m(\vv-\Delta-e_i, b)\\
&\qquad + \sum_{\Delta \in \hc{n}} \sum_{i, j, k} m(\Delta-e_k, a)\, m(\Delta-e_j, a) m(\vv-\Delta-e_i, b)\\
&\qquad = \sum_{\Delta \in \hc{n}} m''(\Delta, a)\, m(\Delta, a) m'(\vv-\Delta, b) + m'(\Delta, a)\, m'(\Delta, a) m'(\vv-\Delta, b)\\
&\qquad = \frac{1}{2} \frac{\partial^3}{\partial a^2\, \partial b} \sum_{\Delta \in \hc{n}} m(\Delta, a)^2 m(\vv-\Delta, b)\\
&\qquad = \frac{1}{2} \frac{\partial^3}{\partial a^2\, \partial b} \exp\left(n G_{\frac{k}{n}}(a, b)\right).\\
\end{align*}
Similarly, using the fact that all derivatives of $m(v, t)$ are non-negative, we have
\begin{align*}
&\sum_{\Delta \in \hc{n}} \sum_{i, j, k, l} m(\Delta-e_l-e_j, a)\, m(\Delta-e_k, a) m(\vv-\Delta-e_i, b)\\
&\qquad \leq \frac{1}{6} \frac{\partial^4}{\partial a^3\,\partial b} \exp\left(n G_{\frac{k}{n}}(a, b)\right),
\end{align*}
and
\begin{align*}
&\sum_{\Delta \in \hc{n}} \sum_{i, j, k} m(\Delta-e_k-e_i, a)\, m(\Delta-e_j, a) m(\vv-\Delta, b)\\
&\qquad \leq \frac{1}{6} \frac{\partial^3}{\partial a^3} \exp\left(n G_{\frac{k}{n}}(a, b)\right).
\end{align*}

Let $c$ denote the minimum of $\frac{1}{2} \cosh 2a - \frac{1}{2} e^{-2b}$ over all $a, b \geq 0$ such that $\varepsilon \leq a+b \leq 1$. It is clear that $c>0$. This means that for any $a, b$ in this range and any $0 \leq x \leq 1$ we have
\begin{equation}
\abs{\frac{\partial}{\partial a} G_x(a, b)} = \abs{x \frac{\sinh 2a}{ \frac{1}{2}\cosh 2a - \frac{1}{2}e^{-2b}} + (1-x) \frac{\sinh 2a}{ \frac{1}{2}\cosh 2a + \frac{1}{2}e^{-2b}}} \leq c^{-1} \sinh 2a.
\end{equation}
Hence, for sufficiently large $C>0$ we have $\abs{\frac{\partial}{\partial a} G_x(a, b)} \leq C\, a$ whenever $0\leq x \leq 1$ and $a, b \geq 0$ such that $\varepsilon \leq a+b \leq 1$. Moreover, as $G_x(a, b)$ is smooth wherever it is defined, we know that for $C$ sufficiently large all partial derivatives of order up to $4$ of $G_x(a, b)$ are bounded in absolute value by $C$ when the pair $(a, b)$ is in this domain.

By explicitly writing out the partial derivatives of $\exp\left(n G_{\frac{k}{n}}(a, b)\right)$ above and combining this with Proposition \ref{prop:fugly} we see that \eqref{eq:Bgettingnicer} holds for sufficiently large $C$, as desired.
\end{proof}
For a given sequence $\vv=\vv_n$ as above, we define
\begin{equation}
\begin{split}
f_n(t) &= G_{\frac kn}(t, \vartheta_n-t)\\
&=\vartheta_n-t + \frac{k}{n} \ln\left(\frac{1}{2} \cosh 2t - \frac{1}{2} e^{-2\vartheta_n+2t}\right)\\
&\qquad+ \frac{n-k}{n}\ln\left(\frac{1}{2} \cosh 2t + \frac{1}{2} e^{-2\vartheta_n+2t}\right)
\end{split}
\end{equation}
and 
\begin{equation}
\begin{split}
f(t) &= G_x(t, \vartheta(x)-t)\\
&=\vartheta(x)-t + x \ln\left(\frac{1}{2} \cosh 2t - \frac{1}{2} e^{-2\vartheta(x)+2t}\right)\\ &\qquad+ (1-x)\ln\left(\frac{1}{2} \cosh 2t + \frac{1}{2} e^{-2\vartheta(x)+2t}\right).
\end{split}
\end{equation}
Note that $f$ depends on $x$. From the definition of $G_x(a, b)$ we see that $f_n(0)=-\frac{\ln n}{n}$ and $f_n(\vartheta_n)=-2\frac{\ln n}{n}$, and that $f(0)=f(\vartheta(x))=0$, see \eqref{eq:defofvartheta}.

Suppose that $f_n(t)$ is ``asymptotically U-shaped'' in the sense that exists a constant $\lambda>0$ such that for sufficiently large $n$ we have 
\begin{equation}
f_n(t) \leq \max\left(f_n(0)-\lambda t, f_n(\vartheta_n) - \lambda (\vartheta_n-t)\right)
\end{equation}
for any $0 \leq t \leq \vartheta_n$. If this holds, then by Proposition \ref{prop:Bgettingnicer} we have
\begin{equation}
B(\vv_n, \vartheta_n) \leq \int_0^\infty O\left(n^3 t^3+ n^2 t + n\right) e^{-\lambda n t}\,dt,
\end{equation}
which would imply that $B(\vv_n, \vartheta_n)=O(1)$ as desired. It remains to show for which sequences of vertices $\vv=\vv_n$, $f_n$ is asymptotically U-shaped. We start by giving a simple sufficient condition for $x$.
\begin{proposition}\label{prop:easyboundx}
Suppose that $x>1-\frac{\ln (2\sqrt{2})}{\ln 3}\approx 0.054$. Then $f_n(t)$ is asymptotically U-shaped.
\end{proposition}
\begin{proof}
By some straight-forward but tedious calculations we see that
\begin{equation}
f_n''(t) = \frac{k}{n}  \frac{4(1 - 2 e^{-2\vartheta_n})}{\left( \cosh 2t - e^{-2\vartheta_n+2t}\right)^2} + \frac{n-k}{n} \frac{4(1 + 2 e^{-2\vartheta_n})}{\left( \cosh 2t + e^{-2\vartheta_n+2t}\right)^2}.
\end{equation}
Now, if we assume that $\vartheta_n \geq \ln \sqrt{2}$, then the first term in the right-hand side is non-negative, and so we have that $f_n''(t)$ is at least, say, $\frac{n-k}{100 n}$ for all $0 \leq t \leq \vartheta_n$. It follows that if $\vartheta(x)>\ln\sqrt{2}$, then $f_n(t)$ is asymptotically U-shaped. The proposition follows by the easily verified fact that $\vartheta\left(1-\frac{\ln (2\sqrt{2})}{\ln 3}\right) = \ln\sqrt{2}$.
\end{proof}
It is clear from the proof of Proposition \ref{prop:easyboundx} that the limit $1-\frac{\ln(2\sqrt{2})}{\ln 3}$ is not optimal, and can be lowered by considering $f_n''(t)$ more closely. It turns out however that there is a limit for $x$ at which the convexity of $f_n$ breaks down, and more importantly for sufficiently small $x$ the asymptotic U-shape of $f_n$ breaks down. In the remaining part of this section, we will investigate when this occurs.

By some more straight-forward but tedious calculations we see that
\begin{equation}\label{eq:fnprime}
\begin{split}
&f_n'(t) \cdot \left(\cosh 2t - e^{-2\vartheta_n + 2t} \right)\left(\cosh 2t + e^{-2\vartheta_n + 2t} \right) \\
&=\left(\frac{1}{4}-e^{-4\vartheta_n}\right)e^{4t}-\frac{3}{4} e^{-4t} - \frac{1}{2} + 2 \frac{n-2k}{n} e^{-2\vartheta_n}.
\end{split}
\end{equation}
This expression has the same sign as $f_n'(t)$. We see that depending on the sign of $\frac{1}{4}-e^{-4\vartheta_n}$ it is either increasing or concave, hence $f_n$ changes sign at most twice. Furthermore, if $f_n$ changes sign twice it goes from negative to positive to negative. In the same way, since
\begin{equation}\label{eq:fprime}
\begin{split}
&f'(t) \cdot \left(\cosh 2t - e^{-2\vartheta(x) + 2t} \right)\left(\cosh 2t + e^{-2\vartheta(x) + 2t} \right) \\
&=\left(\frac{1}{4}-e^{-4\vartheta(x)}\right)e^{4t}-\frac{3}{4} e^{-4t} - \frac{1}{2} + 2 (1-2x) e^{-2\vartheta(x)}.
\end{split}
\end{equation}
the same must be true for $f(t)$.

Combining this observation with the fact that $f_n''(t)$ is bounded it follows that a necessary and sufficient condition for $f_n$ being asymptotically U-shaped is that $\lim_{n\rightarrow\infty} f_n'(0) = f'(0) < 0$ and $\lim_{n\rightarrow\infty} f_n'(\vartheta_n) = f'(\vartheta(x)) > 0$. In fact, the former condition is implied by the latter as then $f_n'(t)$ changes sign at most once, but $f(0)=f(\vartheta(x))=0$.

\begin{figure}
\centering
\includegraphics[width=\textwidth]{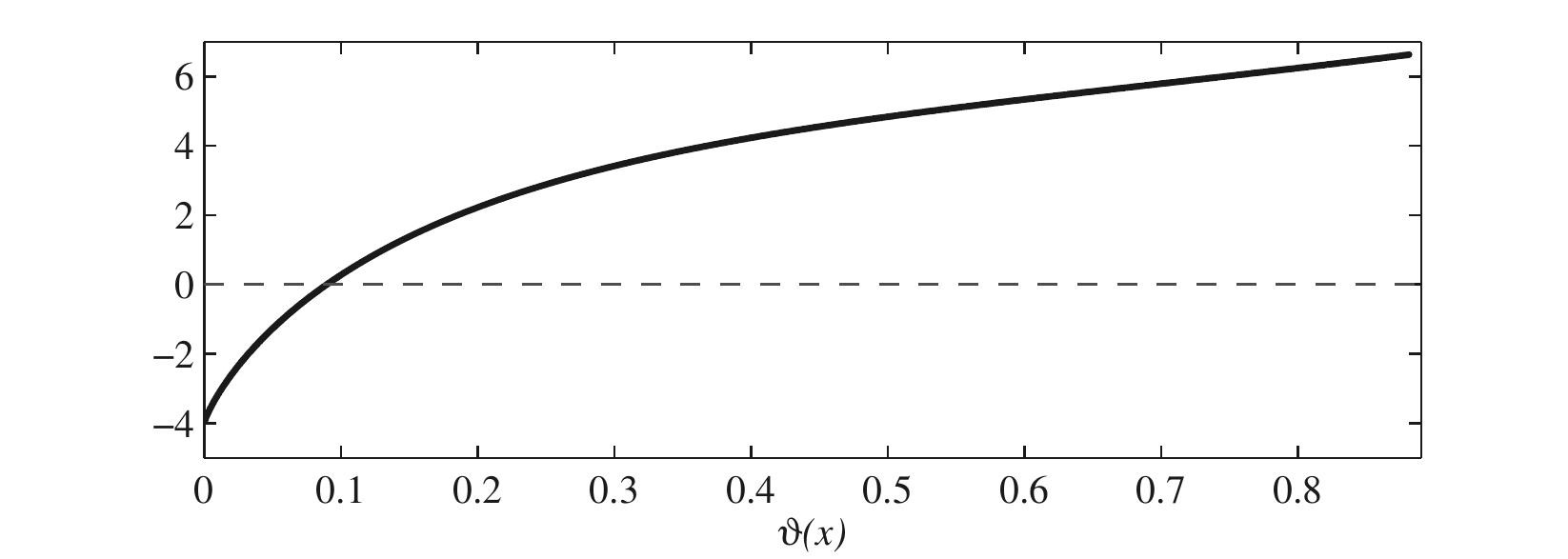}
\caption{Plot of equation \eqref{eq:bastardizedf} divided by $x$ as a function of $\vartheta(x)$. We see that as $\vartheta$ tends to $0$, this converges to its limit of $-4$. The curve intersects the $\vartheta$-axis at $\vartheta(x)\approx 0.0898$, that is at $x\approx 0.00167$.}\label{fig:fuglymatlab}
\end{figure}

As $1=\left(\cosh \vartheta \right) \left(\tanh \vartheta\right)^x$ we have
\begin{equation}
x(\vartheta) = \frac{\ln \cosh\vartheta}{-\ln \tanh \vartheta} = \frac{\vartheta^2}{-2\ln \vartheta} + O\left(\frac{\vartheta^4}{(\ln \vartheta)^2}\right).
\end{equation}
Hence, we have an explicit expression for \eqref{eq:fprime}  as a function of $\vartheta(x)$. Plugging $t=\vartheta(x)$ into the right-hand side of this expression we get
\begin{equation}\label{eq:bastardizedf}
\frac{1}{4}e^{4\vartheta} -\frac{3}{4} e^{-4\vartheta} + 2 (1-2x) e^{-2\vartheta}- \frac{3}{2}.
\end{equation}
Note that this has the same sign as $f'(\vartheta(x))$. By Taylor expanding this expression in $x$ and $\vartheta$ we see that the dominating term for small $\vartheta$ is $-4x$. Hence $f_n$ is not asymptotically U-shaped for sufficiently small $x$. To get a picture of what happens when $x$ increases, we divide \eqref{eq:bastardizedf} by $x$ and plot as a function of $\vartheta$, see Figure \ref{fig:fuglymatlab}. It is clear that there is a critical value $x^*$ slightly less than $0.0017$ such $f_n$ is asymptotically U-shaped if and only if $x>x^*$. This proves the following proposition:
\begin{proposition}\label{prop:boundB}
Let $\{\vv_n\}_{n=1}^\infty$ be a sequence of vertices, $\vv_n \in \hc{n}$, such that $\lim_{n\rightarrow\infty} \abs{\vv_n}/n$ exists and is strictly greater than $x^*$. Then for $\vartheta_n$ as defined in \eqref{eq:defvarthetan} we have $B(\vv_n, \vartheta_n) = O(1)$.
\end{proposition}

\begin{remark}\label{rem:breakdown}
Throughout this section we have only really been interested in deriving a tractable upper bound for $B(\vv_n, \vartheta_n)$ without discussing sharpness. Nevertheless, it is not too hard to convince oneself that the bound given in Proposition \ref{prop:Bgettingnicer} is sharp up to, say, a polynomial factor in $n$. However, for $x < x^*$ we know that there exists an interval of positive length for $t$ where $f_n(t)$ is positive, which would then imply that $B(\vv_n, \vartheta_n)$ diverges exponentially fast in $n$.
\end{remark}

\section{Completing the proof of Theorem \ref{thm:mainresult1equiv}}\label{sec:bootstrap}

Let $\{\vv_n\}_{n=1}^\infty$ be a sequence of vertices, $\vv_n \in \hc{n}$ for each $n$, such that $x=\lim_{n\rightarrow\infty} \abs{\vv_n}/n$ exists and is at least $0.002$ and let $\{\vartheta_n\}_{n=1}^\infty$ be as in \eqref{eq:defvarthetan}. Applying the estimates of $S(\vv_n, \vartheta_n)$ and $B(\vv_n, \vartheta_n)$ from Propositions \ref{prop:boundS} and \ref{prop:boundB} to Theorem \ref{thm:uncontestedCTP} it follows by Proposition \ref{prop:fppisaliveCTP} and Lemma \ref{lemma:uncontestedisalive} that there exists a constant $c_{0}>0$ such that
\begin{equation}
\liminf_{n\rightarrow\infty} \mathbb{P}\left(\mathcal{T}_V'(\vz, \vv_n) \leq \vartheta_n\right) \geq c_{0}.
\end{equation}
Since $\vartheta_n \rightarrow \vartheta(x)$ as $n\rightarrow\infty$, this means in particular that for any $\varepsilon>0$ we have
\begin{equation}\label{eq:bootstrapstart}
\liminf_{n\rightarrow\infty} \mathbb{P}\left(\mathcal{T}_V'(\vz, \vv_n) \leq \vartheta(x)+\varepsilon\right) \geq c_0.
\end{equation}
Note that we can assume that $c_0$ is independent of the choice of sequence.

\begin{proposition}
Let $\{\vv_n\}_{n=1}^\infty$ be a sequence as above, and let $x=\lim \abs{\vv}/n$. Then, for any $\varepsilon>0$ we have
\begin{equation}
\mathbb{P}\left( \mathcal{T}_V'(\vz, \vv_n) \leq \vartheta(x) + \varepsilon\right)\rightarrow 1
\end{equation}
as $n\rightarrow\infty$.
\end{proposition}
\begin{proof}
Let $\varepsilon >0$ be arbitrary. Condition on the vertex passage times of all neighbors of $\vz$ and $\vv_n$. Assuming $\abs{\vv_n}\geq 3$, it is easy to see that the number of coordinate places $1\leq i \leq n$ with the property that the $i$:th coordinate of $\vv_n$ is 1, and the cost of both $e_i$ and $\vv_n-e_i$ are at most $\varepsilon/3$, is distributed as $\operatorname{Bin}\left(\abs{\vv_n}, (1-e^{-\varepsilon/3})^2\right)$. Hence as $n\rightarrow\infty$ it is clear that, with probability $1-o(1)$, there are at least two such coordiantes. Pick a pair $i\neq j$.

Depending on the choice of $i$ and $j$, we define $Q_0$ as the induced subgraph of $\hc{n}$ with vertex set $\{v \in \hc{n} : v_i=1, v_j=0\}$. We similarly define $Q_1$ as the induced subgraph of $\hc{n}$ with vertex set $\{v \in \hc{n} : v_i=0, v_j=1\}$. Note that $Q_0$ and $Q_1$ are vertex disjoint subgraphs of $\hc{n}$, both isomorphic to $\hc{n-2}$.

In light of $Q_0$ and $Q_1$, we have two natural upper bounds for $\mathcal{T}_V'(\vz, \vv_n)$, namely $c(e_i)+c(\vv_n-e_j)$ plus the smallest reduced vertex passage time for any path from $e_i$ to $\vv_n-e_j$ in $Q_0$, and $c(e_j)+c(\vv_n-e_i)$ plus the smallest reduced vertex passage time for any path from $e_j$ to $\vv_n-e_i$ in $Q_1$. As the only vertices of $Q_0$ and $Q_1$ which are neighbors of $\vz$ or $\vv_n$ are $e_i$, $e_j$, $\vv_n-e_i$ and $\vv_n-e_j$, the reduced first-passage times in $Q_0$ and $Q_1$ are independent of each other and each is distributed as the reduced first-passage time between two vertices at distance $\abs{\vv_n}-2$ in $\hc{n-2}$. By applying \eqref{eq:bootstrapstart} to the first-passage percolation problems in $Q_0$ and $Q_1$, we conclude that for any $\varepsilon > 0$ and for any sequence $\{\vv_n\}_{n=1}^\infty$  where $\vv_n \in \hc{n}$ for each $n\geq 1$ such that $x=\lim_{n\rightarrow\infty} \abs{\vv_n}/n$ exists and is at least $0.002$, we have
\begin{equation}\label{eq:bootstraponestep}
\liminf_{n\rightarrow\infty} \mathbb{P}\left( \mathcal{T}_V'(\vz, \vv_n) \leq \vartheta(x) + \varepsilon\right) \geq 1-(1-c_0)^2.
\end{equation}
Note that this is the same expression as \eqref{eq:bootstrapstart}, except that the right-hand side here is strictly larger. Hence, by iteratively applying this argument, we see that we can replace the right-hand side in \eqref{eq:bootstraponestep} by $c_k=1-(1-c_{0})^{2^k}$ for any non-negative integer $k$. The Proposition follows by letting $k\rightarrow\infty$.
\end{proof}

\section*{Acknowledgements}
The author would like to thank his supervisor, Peter Hegarty, for helpful discussions.

\end{document}